\documentclass{article}
\usepackage[utf8]{inputenc}

\title{An algorithm for computing the $\Upsilon$-invariant and the $d$-invariants of Dehn surgeries}

\author{Taketo Sano and Kouki Sato}

\usepackage{amsfonts}
\usepackage{amssymb}
\usepackage{amsmath}
\usepackage{amsthm}
\usepackage{mathtools}
\usepackage{stmaryrd}
\usepackage{cleveref}
\usepackage{url}
\usepackage{float}
\usepackage{adjustbox}
\usepackage{booktabs}
\usepackage[nottoc]{tocbibind}
\usepackage{subcaption}
 \usepackage{mathrsfs}
 \usepackage{mathtools}

\usepackage{tikz}
\usepackage{tikz-cd}
\usepackage{amscd}
\usepackage{multicol}

\usetikzlibrary{patterns}

\theoremstyle{plain}
    \newtheorem{proposition}{Proposition}[section]
    \newtheorem{theorem}[proposition]{Theorem}
    \newtheorem{corollary}[proposition]{Corollary}
    \newtheorem{lemma}[proposition]{Lemma}

\theoremstyle{definition}
    \newtheorem{definition}[proposition]{Definition}

\theoremstyle{remark}
	\newtheorem{remark}[proposition]{Remark}%

    \def\G{\mathcal{G}}
    \def\tG{\widetilde{\mathcal{G}}}
    \def\mF{\mathcal{F}}
    \def\z{\mathbb{Z}}
    \def\CR{\mathcal{CR}}
    
    \def\F{\mathbb{F}}

    \def\Z{\mathbb{Z}}
    \def\R{\mathbb{R}}
    \def\Q{\mathbb{Q}}
    \def\np{\nu^+}
    \def\nuplus{\overset{\np}{\sim}}
    \def\mC{\mathcal{C}}

    \DeclareMathOperator{\Alex}{Alex}
    \DeclareMathOperator{\Alg}{Alg}
    \DeclareMathOperator{\gr}{gr}
    \DeclareMathOperator{\Hom}{Hom}
    
    \DeclareMathOperator{\Spin}{Spin}

    \DeclareMathOperator{\PL}{PL}
    \newcommand{\falex}[1]{\mF^{\Alex}_{#1}}
    \newcommand{\falg}[1]{\mF^{\Alg}_{#1}}
       
    \newcommand{\heq}[2]{\underset{#1\text{-filt.}/#2}{\simeq}}

    \def\x{\mathbf{x}}
    \def\y{\mathbf{y}}
    \DeclareMathOperator{\shift}{shift}

\begin{document}

    \maketitle
    
    \begin{abstract}
    By using grid homology theory, we give an explicit algorithm for computing Ozsv\'ath-Stipsicz-Szab\'o's $\Upsilon$-invariant and the $d$-invariant of Dehn surgeries along knots in $S^3$. As its application, we compute the two invariants for all prime knots with up to 11 crossings. 
    \end{abstract}
    
    \section{Introduction} \label{sec:intro}
The {\it $\Upsilon$-invariant} 
defined by Ozsv\'ath, Stipsicz and Szab\'o \cite{OSS17}
is a group homomorphism
\[
\Upsilon \colon \mC \to \PL([0,2], \R),
\]
where $\mC$ denotes the smooth knot concordance group 
and $\PL([0,2], \R)$ the vector space of piecewise-linear functions on $[0,2]$. 
The $\Upsilon$-invariant is known as a very effective tool in the study of smooth knot concordance. Indeed, the slope of $\Upsilon$ gives a
surjective homomorphism
$\mC \to \Z^\infty$ whose image is generated by topologically slice knots, and this proves that the subgroup of $\mC$ generated by topologically slice knots contains a direct summand isomorphic to $\Z^\infty$.
For a given knot $K$ in $S^3$, we denote the $\Upsilon$-invariant of $K$ by $\Upsilon_K$.

The {\it $d$-invariant} defined by Ozsv\'ath and Szab\'o \cite{OS03grading} is a group homomorphism
\[
d \colon \theta^c \to \Q,
\]
where $\theta^c$ denotes the $\Spin^c$ 
rational homology cobordism group of $\Spin^c$ rational homology 3-spheres (see \cite[Definition 1.1]{OS03grading} for the precise definition of $\theta^c$).
The $d$-invariant is also known as a very powerful invariant so that
it enables us to reprove Donaldson's diagonalization theorem.
Moreover, the $d$-invariants of 
Dehn surgeries and prime-powered branched covers along knots in $S^3$ 
yield an infinite family of smooth knot concordance invariants
(see \cite{Hom17} for a survey describing these facts and studies of them). 
For any knot $K$ in $S^3$ and coprime integers $p$ and $q$ with $p \neq 0$ and $q>0$, let $S^3_{p/q}(K)$ denote the $p/q$-surgery along $K$. 
In this paper, we focus on the $d$-invariant of $S^3_{p/q}(K)$.
\footnote{While $0$-surgeries are not rational homology 3-spheres, the $d$-invariants of them are also defined. For the cases of knots in $S^3$, they are determined by $d(S^3_{\pm 1}(K))$. See \cite[Section 4.2]{OS03grading}.}
Note that since $S^3_{p/q}(K)$ has $|p|$ different $\Spin^c$ structures, the $d$-invariant of $S^3_{p/q}(K)$ is obtained as a $|p|$-tuple of rational numbers. We denote the $|p|$-tuple by $d(S^3_{p/q}(K))$.

The invariants $\Upsilon_K$ and $d(S^3_{p/q}(K))$ were originally defined in different packages of Heegaard Floer theory, but later Livingston \cite{Liv17} and Ni-Wu \cite{NW15} have translated them 
into the words of the doubly filtered chain complex $CFK^{\infty}(K)$ defined in \cite{OS04knot}. 
Moreover, 
$CFK^{\infty}(K)$ 
has a combinatorial description
called {\it grid homology theory} introduced by
Manolescu, Ozsv\'ath and Sarkar \cite{MOS09grid1}.
Therefore, it is natural to consider computing $\Upsilon_K$ and $d(S^3_{p/q}(K))$ 
algorithmically by using grid homology theory.
However, while $CFK^\infty(K)$ is freely and finitely generated over the Laurent polynomial $\F[U^{\pm 1}]$ (where $\F:=\Z/2\Z$), the combinatorial chain complex $C^\infty(G)$ corresponding to $CFK^{\infty}(K)$ is freely and finitely generated over $\F[U^{\pm 1}, U^{\pm 1}_2, \ldots, U_n^{\pm 1}]$ 
with $n>1$. 
In particular,
$C^\infty(G)$ has infinite rank over $\F[U^{\pm 1}]$,
and hence 
there is no obvious way to
use $C^\infty(G)$ for 
algorithmic computation of $\Upsilon_K$ and $d(S^3_{p/q}(K))$.
In this paper, we provide a method for avoiding 
the infinite-dimensional problem, and give an explicit algorithm 
for computing $\Upsilon_K$ and $d(S^3_{p/q}(K))$ of any knot $K$ in $S^3$ and $p/q \in \Q$.
Moreover, as its applications, we complete the lists of
\begin{itemize}
    \item the $\Upsilon$-invariants of  all prime knots with up to 11 crossings, 
    \item all $d$-invariants of all Dehn surgeries along all prime knots with up to 11 crossings, and
    \item all prime knots with up to 12 crossings whose $\Upsilon$ and $d$-invariants coincide with those of the unknot.
\end{itemize}

\subsection{The Key idea}
\label{subsec: key ideas}
In order to give an algorithm for computing 
$\Upsilon_K$
and $d(S^3_{p/q}(K))$, 
we will translate these invariants several times.
The key point is to use the second author's invariant $\G_0(K)$ \cite{2019arXiv190709116S},
from which we can easily compute both $\Upsilon_K$
and $d(S^3_{p/q}(K))$. While $\G_0(K)$ was also defined using $CFK^\infty(K)$, we show in this paper that it
can be translated into the words of the subcomplex 
$CFK^-(K)$ over $\F[U]$. The combinatorial chain complex $C^-(G)$ corresponding to $CFK^-(K)$ is finitely and freely generated by $\F[U, U_2, \ldots, U_n]$ and so 
it has infinite rank over $\F[U]$, but as a graded $\F$-vector space, each degree of $C^-(G)$ has finite rank. (It does not hold for $C^\infty(G)$, i.e.\ each degree of $C^\infty(G)$ has infinite rank over $\F$.)
This fact enables us to reduce the problem to 
finite $\F$-linear systems, which makes $\G_0(K)$ algorithmically computable.

\subsection{Observations of computation results}
\label{subsec: observations}

Here we show several observations of our computation results.
We first state the relationship between
the {\it $\tau$-invariants} \cite{OS03tau}
and the {\it $\np$-equivalence classes} 
\cite{Hom17, KP18} of knots with up to 12 crossings.
Here, the {\it $\tau$-invariant} is a group homomorphism $\mC \to \Z$  
which is already computed for all knots with up to 11 crossings \cite{BG12}.
(Now, the list of the $\tau$-invariants for knots with up to 12 crossings is available in KnotInfo~\cite{Knotinfo}.)
The {\it $\np$-equivalence} is an equivalence relation 
on $\mC$ whose quotient set can be regarded as a quotient group $\mC_{\np}$ of $\mC$.
It is proved in $\cite{Hom17}$ that a knot $K$ is $\np$-equivalent to the unknot if and only if $CFK^\infty(K)$ is filtered chain homotopic to $CFK^\infty(O) \oplus A$, where $O$ denotes the unknot and $A$ is an acyclic complex, i.e.\ $H_*(A)=0$. 
Moreover, all of $\tau$, $\Upsilon_K$ and $d(S^3_{p/q}(K))$ are invariant under $\np$-equivalence.
In particular, the triviality of $K$ in $\mC_{\np}$ implies $\tau(K)=0$, while the converse does not hold in general.
(For instance, the composite knot $T_{2,7}\# (T_{3,4})^*$ is such an example, where $T_{p,q}$ denotes the positive
$(p,q)$ torus knot and $K^*$ the mirror image of $K$.)
On the other hand, it is also known that the converse holds for all quasi-alternating knots \cite{Pet13} and all genus one knots \cite{2019arXiv190709116S}. 
As an observation of our computation results, we show that the converse also holds for any prime knot with up to 12 crossings. 

\begin{theorem}
\label{thm: converse}
For any prime knot $K$ with up to 12 crossings, 
$K$ is $\np$-equivalent to the unknot 
if and only if $\tau(K)=0$.
\end{theorem}

Here we also mention that the triviality of a knot in $\mC_{\np}$
is determined by the diffeomorphism type of its $0$-surgery.
(The proof immediately follows from the original definition of $\np$-equivalence in \cite{KP18} with \cite[Proposition 1.6]{NW15} and \cite[Proposition 4.12]{OS03grading}.)
\begin{proposition}
If two knots $K$ and $K'$ have the same 0-surgeries,
then $K$ is $\np$-equivalent to the unknot
if and only if
$K'$ is $\np$-equivalent to the unknot.
\end{proposition}
\begin{remark}
The $\np$-equivalence class of a knot is not an invariant of its $0$-surgery in general. Indeed, Yasui \cite{Yas15pre} provides an infinite family of pairs of knots with the same 0-surgeries one of whose $\tau$ is greater than the other by 1. 
\end{remark}
It is meaningful to compare the above propositions 
with Piccirillo's proof of non-sliceness of the Conway knot $11n34$ \cite{Pic18pre}. 
To prove the non-sliceness of the Conway knot, Piccirillo consider the {\it knot trace}, which is a 4-manifold obtained by attaching a 0-framed 2-handle to $B^4$ along a given knot.
It is known that if two knots have the same knot traces, then the sliceness of them coincides.
Piccirillo finds a knot $K'$ whose knot trace is diffeomorphic to that of
the Conway knot, and she shows $s(K') \neq 0$, where $s$ denotes the Rasmussen's homomorphism $s \colon \mC \to 2\Z$ \cite{rasmussen2010}.
Here we note that sharing the knot trace implies sharing the 0-surgeries. Moreover,
since the Conway knot has 11 crossings and trivial $\tau$,
Theorem~\ref{thm: converse} implies that the Conway knot is $\np$-equivalent to the unknot.
Therefore, in contrast to the $s$-invariant, we have the following:
\begin{theorem}
For any knot $K$, if 
the 0-surgery along $K$
is diffeomorphic to that of the Conway knot, then $K$
is $\np$-equivalent to the unknot. In particular,
all of $\tau(K)$, $\Upsilon_K$ and $d(S^3_{p/q}(K))$
coincide with those of the unknot.
\end{theorem}
In particular, we cannot replace $s(K')$
in Piccirillo's proof
with any of $\tau(K')$, $\Upsilon_{K'}$ and
$d(S^3_{p/q}(K'))$.

\subsection*{Organization}
In Section~\ref{sec: review of G_0}, we review the invariant $\G_0(K)$. In Section~\ref{sec: translate G_0}, we translate $\G_0(K)$ into the words of $CFK^-(K)$.
In Section~\ref{sec: translating G_0 to C^-(G)},
we review the grid complex $C^-(G)$, and translate $\G_0(K)$
into $C^-(G)$.
In Section~\ref{sec: algorithm}, we give an algorithm for computing $\G_0(K)$ from $C^-(G)$. 
In Section~\ref{sec: results}, we show our computation results.

This paper also contains several appendices. In Appendix~\ref{sec: duality}, we give the duality theorem for $\G_0(K)$ with respect to the mirroring of $K$, which proves that $\G_0$ for the mirror $K^*$ is algorithmically determined from $\G_0(K)$. This is used for determining $d(S^3_{p/q}(K))$ where $p/q$ is negative.
In Appendix~\ref{subsec:sparse-linear-system}, we discuss methods for handling sparse linear systems. 

\subsection*{Acknowledgements}
The authors thank the first author's supervisor Mikio Furuta for his support. 
Development of the computer program and its computation on a cloud computing platform is supported by JSPS KAKENHI Grant Number 17H06461.  
The authors also thank Charles Bouillaguet for many helpful suggestions about sparse linear systems, \textit{omochimetaru} and the members of \textit{swift-developers-japan} discord server for helpful advice about the implementation of the program.
The first author thanks members of his \textit{academist fanclub}\footnote{\url{https://taketo1024.jp/supporters}} for their support.
The second author was supported by JSPS KAKENHI Grant Number 18J00808. 
    \section{A review of the invariant $\G_0$}
\label{sec: review of G_0}
In this section, we review the invariant $\G_0(K)$ defined in \cite{2019arXiv190709116S}, 
which determines both $\Upsilon_{K}$ and $d(S^3_{p/q}(K))$. 

\subsection{Knot complexes, $CFK^{\infty}$}
\label{subsec: formal knot complexes}

\subsubsection{Poset filtered chain complexes}
\label{poset filtered chain complexes}
Let $P$ be a {\it poset}, i.e.\ a set $P$ with partial order $\leq$.
Then, a \textit{lower set} $R \subset P$ is a subset
such that for any $x \in P$, if there exists an element $y \in R$ satisfying $x \leq y$, then $x \in R$. 
In this paper, we mainly consider the partial order $\leq$ on $\z^2$ given by
$(i,j)\leq(k,l)$ if $i \leq k$ and $j \leq l$.
We specially call a lower set R $\subset$ $\Z^2$ with respect to $\leq$ a {\it closed region}.


Let $\F := \z /2\z$ and $\Lambda$ be an $\F$-algebra.
In this paper, we say that
$(C,\partial)$ is a {\it chain complex $C$ over $\Lambda$}
if $(C,\partial)$ satisfies the following:
\begin{itemize}
\item $C$ is a $\Lambda$-module and $\partial \colon C \to C$ is a
$\Lambda$-linear map with $\partial \circ \partial = 0$. 
\item As an $\F$-vector space, $C$ is decomposed into $\bigoplus_{n \in \z} C_n$,
where each $C_n$ satisfies $\partial(C_n) \subset C_{n-1}$.
\end{itemize}
(Note that the $\Lambda$-action does not preserve the grading in general.)
We often abbreviate $(C,\partial)$ to $C$.
Moreover, we say that $C$ is \textit{$P$-filtered} if 
a subcomplex $C_R$ of $C$ over $\F$ is associated to each closed region $R \subset P$
so that if $R \subset R'$ then $C_R \subset C_{R'}$. 
(Here we remark that $C_R$ is not a $\Lambda$-submodule of $C$ in general.)
We call the set $\{C_R\}_{R \in \CR(P)}$ a {\it $P$-filtration} on $C$.
For instance, a $\Z$-filtration of $C$ can be identified as an ascending filtration in the usual sense, under the identification of an integer $i$ with a closed region $\{m \in \Z \mid m \leq i\}$.
Moreover, 
given two $\z$-filtrations $\{\mF^1_i\}_{i \in \z}$ and 
$\{\mF^2_j\}_{j \in \z}$ on $C$, the set
$$
\{C_R\}_{R \in \CR(\z^2)} := \{\sum_{(i,j)\in R} \mF^1_i \cap \mF^2_j\}_{R \in \CR(\z^2)}
$$
defines a $\z^2$-filtration on $C$.
We call it {\it the $\z^2$-filtration induced by the ordered pair} 
$(\{\mF^1_i\}_{i \in \z}, \{\mF^2_j\}_{j \in \z})$. 
When $C$ is endowed with such an induced $\z^2$-filtration, 
we denote by $C^r$ the $\z^2$-filtered chain complex with the same underlying complex $C$ endowed with the $\z^2$-filtration induced by $(\{\mF^2_i\}, \{\mF^1_j\})$.
We call $C^r$ {\it the reflection of $C$}.

For any two $P$-filtered chain complexes $C$ and $C'$, a map 
$f: C \to C'$ is said to be \textit{$P$-filtered}
if $f(C_R) \subset C'_R$ for any closed region $R$.
Two $P$-filtered chain complexes $C$ and $C'$ are 
\textit{$P$-filtered homotopy equivalent over $\Lambda$}
(and denoted $C \heq{P}{\Lambda} C'$)
if there exists a chain homotopy equivalence map $f: C \to C'$ over $\Lambda$, a chain homotopy inverse of $f$
and their chain homotopies
such that all of the four maps are $P$-filtered
and graded. Such $f$ is called 
a \textit{$P$-filtered homotopy equivalence map over $\Lambda$}.
Particularly, we call the above $f$ a 
\textit{$P$-filtered chain isomorphism over $\Lambda$} if $f$ is a chain isomorphism. It is obvious that 
if $C \heq{P}{\Lambda} C'$, then $C \heq{P}{\Lambda'} C'$
for any subalgebra $\Lambda' \subset \Lambda$.
Moreover, it is easy to show the following lemma.
\begin{proposition}
\label{exact}
Let $C$ and $C'$ be $P$-filtered chain complexes.
If $C \heq{P}{\Lambda} C'$, then for any closed regions $R \subset R'$,
we have an isomorphism between the long exact sequences of $\Lambda'$-modules:
$$
\begin{CD}
\cdots @>{\partial_{*}}>> H_*(C_R) 
@>{i_{*}}>> H_*(C_{R'}) @>{p_{*}}>> H_*(C_{R'}/C_{R}) 
@>{\partial_{*}}>>  \cdots \\
@.   @V{\cong}VV  @V{\cong}VV @V{\cong}VV  @. \\
\cdots @>{\partial_{*}}>> H_*(C'_R) @>{i_{*}}>> H_*(C'_{R'}) 
@>{p_{*}}>> H_*(C'_{R'}/C'_{R}) @>{\partial_{*}}>> \cdots \\
\end{CD}
$$
Here, $i \colon C_R \to C_{R'}$ (resp.\ $p \colon C_{R'} \to C_{R'}/C_R$)
denote the inclusion (resp.\ the projection),
and $\Lambda'$ is the maximal subalgebra of $\Lambda$ so that
all of $C_R$, $C_{R'}$, 
$C'_{R}$ and $C'_{R'}$ are $\Lambda'$-submodules of $C$
and $C'$, respectively.
Moreover, the above isomorphism induces an isomorphism  
between the long exact sequences of graded $\F$-vector spaces:
$$
\begin{CD}
\cdots @>{\partial_{*,n+1}}>> H_n(C_R) 
@>{i_{*,n}}>> H_n(C_{R'}) @>{p_{*,n}}>> H_n(C_{R'}/C_{R}) 
@>{\partial_{*,n}}>>  \cdots \\
@.   @V{\cong}VV  @V{\cong}VV @V{\cong}VV  @. \\
\cdots @>{\partial_{*,n+1}}>> H_n(C'_R) @>{i_{*,n}}>> H_n(C'_{R'}) 
@>{p_{*,n}}>> H_n(C'_{R'}/C'_{R}) @>{\partial_{*,n}}>> \cdots \\
\end{CD}
$$
\end{proposition}

\subsubsection{$CFK^{\infty}$ and $CFK^-$}
\label{subsubsec: CFK}

To each knot $K$ in $S^3$, Ozsv\'ath and Szab\'o \cite{OS04knot} associate 
a $\z^2$-filtered chain complex $CFK^{\infty}(K)$ over 
$\Lambda := \F[U, U^{-1}]$ 
such that if two knots $K$ and $J$ are isotopic, then
$CFK^{\infty}(K) \heq{\z^2}{\Lambda} CFK^{\infty}(J)$.
Let $C := CFK^{\infty}(K)$.
Here we summarize the algrabraic properties:
\begin{enumerate}
\item 
$C$ is a chain complex over $\Lambda$
with decomposition $C= \bigoplus_{n \in \z}C_n$.
The grading of a homogeneous element $x$ is denoted $\gr(x)$
and called \textit{the Maslov grading} of $x$.

\item
$C$ has a $\z$-filtration $\{\falex{j}\}_{j \in \z}$ called \textit{the Alexander filtration}.
The filtration level of an element $x \in C$
is denoted 
$\Alex(x)$ (i.e.\ $\Alex(x) := 
\min \{ j \in \z \mid x \in \falex{j} \}$).

\item
$C$ also has a $\z$-filtration
$\{ \falg{i} \}_{i \in \z}$ called  
{\it the algebraic filtration}. The filtration level of an element $x$ is denoted $\Alg(x)$.
When we regard $C$ as a $\z^2$-filtered complex, we use 
the $\z^2$-filtration induced by the ordered pair 
$(\{\falg{i}\}_{i \in \z}, \{\falex{j}\}_{j \in \z})$.
\item
The action of $U$ lowers the Maslov grading by $2$,
and lowers the Alexander and algebraic 
filtration levels by $1$.
\item
$C$ is a free $\Lambda$-module with finite rank, and there exists a basis $\{x_k\}_{1 \leq k \leq r}$ such that 
\begin{itemize}
\item
each $x_k$ is homogeneous
with respect to the homological grading,
\item
$\falex{0}$ is a free $\F[U]$-module
with a basis $\{ U^{\Alex(x_k)} x_k  \}_{1 \leq k \leq r}$, 
and
\item
$\falg{0}$ is a free $\F[U]$-module
with a basis $\{ U^{\Alg(x_k)} x_k  \}_{1 \leq k \leq r}$. \end{itemize}
We call such $\{x_k\}_{1 \leq k \leq r}$ a \textit{filtered basis}.

\item
There exists a $\z^2$-filtered homotopy equivalence map $\iota: C \to C^r$ over $\Lambda$.

\item
Regard $\Lambda$ as a chain complex with trivial boundary map, and define 
the homological grading by
$$
\Lambda_n =
\left\{
\begin{array}{ll}
\{0, U^{-n/2}\} &(n: \text{ even})\\
0 & (n: \text{ odd})
\end{array}
\right.
$$
and the Alexander and 
algebraic filtrations by
$$
\falex{i}(\Lambda)=\falg{i}(\Lambda) = U^{-i} \cdot \F[U].
$$
Then there exists a $\z$-filtered homotopy equivalence map
$f_{\Alex}$ (resp.\ $f_{\Alg}$) $:C \to \Lambda$ over $\Lambda$ 
with respect to 
the Alexander
(resp.\ algebraic) filtrations.
\end{enumerate}
\begin{remark}
The above seven conditions are used as the axioms of {\it formal knot complexes}, to which we can generalize 
the definitions of the invariants $\tau(K)$, $\Upsilon_K$ and $d(S^3_{p/q}(K))$. For more details, see 
\cite[Section 2]{2019arXiv190709116S}.
\end{remark}
Particularly, the subcomplex $\falg{0}$ is denoted by 
$CFK^-(K)$.
\subsubsection{The dual of $CFK^{\infty}$}
Next, we discuss the following duality theorem for $CFK^{\infty}$, which will be used to prove
the duality theorem for $\G_0$ (\Cref{thm: duality}) and
to compute $d(S^3_{p/q}(K))$ with $p/q < 0$.
\begin{theorem}[\text{\cite[Proposition~3.7]{OS04knot}}]
\label{dual thm}
For any knot $K$ and its mirror $K^*$, 
we have
\[
CFK^\infty(K^*)\heq{\z^2}{\Lambda}(CFK^\infty(K))^*,
\]
where 
\[(CFK^\infty(K))^* := \Hom_{\Lambda}(CFK^\infty(K),\Lambda)
\]
is the dual complex of $CFK^{\infty}(K)$. 
\end{theorem}
\begin{remark}
Precisely,
\cite[Proposition~3.7]{OS04knot}
states the theorem for the knot Floer homology $\widehat{HFK}$,
while the same proof can be applied to 
$CFK^{\infty}$.
\end{remark}

Here we explain how to write the Maslov graiding and 
the Alexander and algebraic filtrations on 
$(CFK^\infty(K))^*$ in the words of $CFK^{\infty}(K)$. 
(The following arguments follow from 
\cite[Subsection 2.4]{2019arXiv190709116S}.)

Let $C := CFK^{\infty}(K)$. 
Define a $\F$-linear map
$\varepsilon : \Lambda \to \F$ by $\varepsilon (p(U)) = p(0)$ for each $p(U) \in \Lambda$
(i.e. $\varepsilon$ maps a Laurent polynomial to its constant term).
Then, 
the Maslov graiding and 
the Alexander and algebraic filtrations on 
$C^*=(CFK^\infty(K))^*$ are given by
\[
C^*_n= \left\{ \varphi \in C^* 
\ \middle|\  \varepsilon \circ \varphi (\bigoplus_{m \neq -n} C_m) = \{ 0\} \right\},
\]
\[
\falex{j}(C^*)= \left\{ \varphi \in C^* 
\ \middle|\  \varepsilon \circ 
\varphi (\falex{-j-1}) = \{ 0\} \right\},
\]
and
\[
\falg{i}(C^*)= \left\{ \varphi \in C^* 
\ \middle|\  \varepsilon \circ 
\varphi (\falg{-i-1}) = \{ 0\} \right\}.
\]
In addition, we also note that
for any filtered basis $\{ x_k \}_{1\leq k \leq r}$
for $C$, 
the dual basis $\{ x^*_k \}_{1\leq k \leq r}$
becomes a filtered basis for $C^*$.

Here we also mention the following lemma.
\begin{lemma}[\text{\cite[Lemma 2.17]{2019arXiv190709116S}}]
\label{lem: dual epsilon}
Let $C := CFK^\infty(K)$.
Then the $\F$-linear map $\varepsilon_n:C^*_{-n} \to \Hom_{\F}(C_n,\F)$
defined by $\varphi \mapsto \varepsilon \circ \varphi$
is a cochain isomorphism (where we see $\{C^*_{-n}\}_{n \in \z}$ as a graded cochain complex over $\F$).
In particular, we have $\F$-linear isomorphisms
$$
H_{-n}(C^*) \cong H^n(C_*;\F) \cong \Hom_{\F}(H_n(C_*),\F),
$$
where the first isomorphism is induced from $\varepsilon_n$.
\end{lemma}

\subsection{$\np$-equivalence}
\label{np equivalence}

The {\it $\np$-equivalence} is an equivalence relation on knots
and regarded as a $CFK^{\infty}$-version of knot concordance.

First, we note that the seventh property of
$CFK^\infty$ implies
$$
H_*(CFK^\infty(K)) \cong \Lambda 
$$
and
$$
H_n(CFK^\infty(K)) \cong 
\begin{cases}
\F & (n \colon \text{even})\\
0 & (n \colon \text{odd})
\end{cases}
$$
for any knot $K$.
(This property is called {\it the global triviality}.)
Here we consider morphisms between formal knot complexes 
in terms of the global triviality.
For two knots $K$ and $J$, a chain map 
$f \colon CFK^\infty(K) \to CFK^\infty(J)$ over $\Lambda$ is a 
{\it $\Z^2$-filtered quasi-isomorphism} 
if $f$ is $\Z^2$-filtered, graded, and induces an isomorphism
$f_* \colon H_*(CFK^\infty(K)) \to H_*(CFK^\infty(J))$. 
Now,  {\it the $\np$-equivalence} (or {\it local equivalence}) is defined as follows.
\begin{definition}
Two knots $K$ and $J$ are
{\it $\np$-equivalent} (and denoted $K \nuplus J$)
if there exist $\Z^2$-filtered quasi-isomorphisms
$f \colon CFK^\infty(K) \to CFK^\infty(J)$ and 
$g \colon CFK^\infty(J) \to CFK^\infty(K)$.
\end{definition}

\begin{remark}
Originally, the $\np$-equivalence is defined by using 
{\it Hom-Wu's $\np$-invariant} \cite{HW16}, which is a $\Z_{\geq 0}$-valued knot concordance invariant. Namely,
two knots $K$ and $J$ are
$\np$-equivalent if and only if the equalities 
$\nu^+(K \# -J^*) = \nu^+(-K^* \# J)=0$ hold.
For more details, see \cite[Section 2]{2019arXiv190709116S}.
\end{remark}
It is obvious that $\nuplus$ is an equivalence relation on 
knots.
We call the equivalence class of a knot $K$ under $\nuplus$
\textit{the $\nu^+$-equivalence class} or \textit{$\nu^+$-class of $K$}, and denote
it by $[K]_{\nu^+}$. The quotient set of knots under $\nuplus$
is denoted by $\mC_{\np}$.
Then, we have the following theorem.
\begin{theorem}[\text{\cite{Hom17}, \cite[Theorem 2.37]{2019arXiv190709116S}}]
If two knots are concordant, then they are $\np$-equivalent.
Moreover, the operation $[K]_{\np} + [J]_{\np} := [K\# J]_{\np}$
endows $\mC_{\np}$ with an abelian group strucutre
so that the surjective map
\[
\mC \to \mC_{\np},\ [K] \mapsto [K]_{\np}
\]
is a group homomorphism.
\end{theorem}
Here we also mention a partial order on $\mC_{\np}$.
For two elements $[K]_{\np},[J]_{\np} \in \mC_{\np}$,
we denote $[K]_{\np} \leq [J]_{\np}$ if
there exists a $\Z^2$-filtered quasi-isomorphism 
$f \colon CFK^\infty(K) \to CFK^\infty(J)$.
Then we see that the relation $\leq$ defines a partial order on $\mC_{\np}$. For the partial order, we have the following 4-genus bound.
\begin{theorem}[\text{\cite[Theorem 1.5]{2019arXiv190709116S}}]
\label{thm:g4-bound}
Let $g_4$ denote the 4-genus of a knot $K$. Then we have
\[
-g_4[T_{2,3}]_{\nu^+} \leq [K]_{\nu^+} \leq g_4[T_{2,3}]_{\nu^+}.
\]
\end{theorem}

    \subsection{The invariant $\G_0$}
\label{the invariant G_0}
Now we recall 
the invariant $\G_0$,
which is given as an invariant of knots under
$\np$-equivalence.
\subsubsection{The invariants $\tG_0$ and $\G_0$}
\label{tG_0 and G_0}
For $C:=CFK^\infty(K)$, a cycle $x\in C$ 
is called a {\it homological generator (of degree 0)} 
if $x$ is homogeneous with $\gr(x)=0$
and the homology class $[x] \in H_0(C)\cong \F$ is non-zero.
Then, we define
$$
\tG_0(K) := \{R \in \CR(\z^2) \mid C_R \text{ contains a homological generator} \}.
$$
Any element $R \in \tG_0(K)$
is called a {\it realizable region of $K$}.
The set $\tG_0(K)$ behaves naturally with respect to 
filtered quasi-isomorphism.
\begin{theorem}[\text{\cite[Thereom~5.1]{2019arXiv190709116S}}]
\label{tG_0 ineq}
If $[K]_{\np} \leq [J]_{\np}$, then $\tG_0(K) \supset \tG_0(J)$.
\end{theorem}

As a corollary, we have the invariance of $\tG_0$ under $\nuplus$.
Here $\mathcal{P}(\CR(\z^2))$ denotes the power set of $\CR(\z^2)$.
\begin{corollary}
\label{tG_0 invariance}
$\tG_0(K)$ is invariant under $\np$-equivalence. In particular,
\[
\tG_0 \colon [K]_{\np} \mapsto \tG_0(K)
\]
is a well-defined map
$
\mC_{\np} \to \mathcal{P}(\CR(\z^2)).
$
\end{corollary}
Here we note that $\tG_0(K)$ is an infinite set for any $K$.
To extract the essential part of $\tG_0$,
we consider the minimalization of $\tG_0$.

For a subset $\mathcal{S} \subset \CR(\z^2)$,
an element $R \in \mathcal{S}$ is {\it minimal in $\mathcal{S}$}
if it satisfies
$$
\text{if } R' \in \mathcal{S} \text{ and } R' \subset R, \text{ then } R'=R.
$$
Define the map 
$$
\min \colon \mathcal{P}(\CR(\z^2)) \to \mathcal{P}(\CR(\z^2))
$$
by
$$
\mathcal{S} \mapsto \{ R \in \mathcal{S} \mid R \text{ is minimal in } \mathcal{S}\}.
$$ 
Now we define $\G_0(K)$ by 
$$
\G_0(K) := \min \tG_0(K).
$$
The invariance of $\G_0$ under 
$\nuplus$ immediately follows from 
Corollary~\ref{tG_0 invariance}. Moreover, $\G_0(K)$ has 
the following nice properties. \begin{theorem}[\text{\cite[Theorem 5.7]{2019arXiv190709116S}}]
\label{thm: G'_0}
$\G_0(K)$ 
is non-empty and finite.
\end{theorem}

\begin{theorem}[\text{\cite[Corollary 5.8]{2019arXiv190709116S}}]
\label{thm: minimalize}
For any closed region $R$, the following holds:
$$
R \in \tG_0(K) \Leftrightarrow \exists R' \in \G_0(K), 
\ R' \subset R.
$$
\end{theorem}

\begin{theorem}[\text{\cite[Theorem 5.16]{2019arXiv190709116S}}]
\label{thm: detect zero}
For any knot $K$,
the following holds:
$$
[K]_{\np} = 0 \Leftrightarrow \G_0(K) = \{ R_{(0,0)}\},
$$
where $R_{(0,0)} := \{(i,j)\in \Z^2 \mid (i, j) \leq (0,0) \}$.
\end{theorem}

Here we also mention the relationship of $\G_0$ to filtered quasi-isomorphism.

\begin{proposition}[\text{\cite[Proposotion 5.9]{2019arXiv190709116S}}] \label{prop:nuplus-ineq}
If $[K]_{\np} \leq [J]_{\np}$, then for any $R' \in \G_0(J)$, there
exists an element $R \in \G_0(K)$ with $R \subset R'$.
\end{proposition}

Next, as a new result on the invariant $\G_0$, we state the duality theorem for $\G_0$ which implies that $\G_0(K)$
and $\G_0(K^*)$ can recover each other.
Here, for a subset $S \subset \Z^2$, we define the subset 
$-S \subset \Z^2$  by 
    \[
        -S := \{ (i, j) \in \Z^2 \mid (-i, -j) \in S \}.
    \]
\begin{theorem}
\label{thm: duality}
For any knot $K$, the equalities
$$
\tG_0(K^*) = \{ R \in \CR(\z^2) \mid 
\forall R' \in \G_0(K),\ R \cap (-R') \neq \varnothing \}
$$
and
$$
\G_0(K^*) = \min\{ R \in \CR(\z^2) \mid 
\forall R' \in \G_0(K),\ R \cap (-R') \neq \varnothing \}
$$
hold. 
\end{theorem}
\begin{remark}
This equality holds for any formal knot complex. Namely, we can prove the theorem purely algebraically.
\end{remark}
\begin{remark}
Combining with \Cref{thm: minimalize}, we can also prove the equality 
\[
\tG_0(K^*) = \{ R \in \CR(\z^2) \mid 
\forall R' \in \tG_0(K),\ R \cap (-R') \neq \varnothing \},
\]
which looks more symmetric with respect to $\tG_0$.
\end{remark}

\Cref{thm: duality} is proved in Appendix~\ref{sec: duality},
where we also provide a method for computing $\G_0(K^*)$ from $\G_0(K)$ algorithmically.

\subsubsection{Relationship to other concordance invariants}
Here we state the relationship of $\G_0(K)$ to the invariants
$\tau(K)$, $d(S^3_{p/q}(K))$ and $\Upsilon_K$.
We first explain the translation of 
$\{d(S^3_{p/q}(K))\}_{p/q \in \Q_{>0}}$
into 
{\it Ni-Wu's $V_k$-sequence}
$\{V_k(K)\}_{k \in \z_{\geq 0}}$ \cite{NW15}.
Note that there is a canonical identification between the set
of $\Spin^c$ structures over $S^3_{p/q}(K)$ 
and $\{ i \mid  0 \leq i \leq p-1\}$. This identification can be
made explicit by the procedure in \cite[Section 4, Section 7]{OS11rational}.
Let $d(S^3_{p/q}(K), i)$ denote the correction term of $S^3_{p/q}(K)$
with  the $i$-th $\Spin^c$ structure ($0 \leq i \leq p-1$). 
\begin{proposition}[\text{\cite[Proposition~1.6]{NW15}}]
\label{V_k d}
For any knot $K$, $p/q > 0$ and $0 \leq i \leq p-1$, 
the equality
$$
d(S^3_{p/q}(K),i) = d(S^3_{p/q}(O),i) - 2 
\max \left\{ V_{\lfloor \frac{i}{q} \rfloor}(K), V_{\lfloor \frac{p+q -1-i}{q} \rfloor}(K) \right\}
$$
holds,
where $O$ denotes the unknot and $\lfloor \cdot \rfloor$ is the floor function.
\end{proposition}
Since we have the orientation-preserving diffeomorphism
\[
S^3_{-p/q}(K) \cong -S^3_{p/q}(K^*),
\]
Proposition~\ref{V_k d} shows that 
the two sequences $\{V_k(K)\}$ and $\{V_k(K^*)\}$ determine
all values of $\{d(S^3_{p/q}(K))\}_{p/q \in \Q}$.

Here we also need to introduce several classes of closed regions.
For each coordinate $(k,l) \in \z^2$, 
we define the {\it simple region $R_{(k,l)}$
with the  corner} $(k,l)$
by
\[
R_{(k,l)}:=\{(i,j) \in \z^2 \mid i \leq k \text{ and } j \leq l \}.
\]
In addition, for each $t \in [0,2]$ and $s \in \R$, we define 
the {\it linear region $R^t(s)$ with respect to $(t,s)$} by
\[
R^t(s) := \left\{(i,j) \in \z^2 \ \middle|\  (1-\frac{t}{2})i + \frac{t}{2}j \leq s \right\}.
\]

Now, the relationship of $\G_0(K)$ to the invariants
$\tau(K)$, $V_k(K)$ and $\Upsilon_K$ is stated as follows.
\begin{theorem}[\text{\cite[Proposition 5.17]{2019arXiv190709116S}}]
\label{thm: G_0 and others}
The invariants $\tau(K)$, $V_k(K)$ and 
$\Upsilon_K$ are determined from
$\G_0(K)$ by the formulas:
\begin{eqnarray*}
\tau(K)&=&\min\{m \in \z \mid 
\exists R \in \G_0(K), R \subset
(\{ i \leq -1\} \cup R_{(0,m)})   \}\\
V_k(K)&=&\min\{m \in \z_{\geq 0} \mid  
\exists R \in \G_0(K), R \subset R_{(m,k+m)} \}\\
\Upsilon_K(t)&=&-2 \left( \min\{s \in \R \mid \exists R \in \G_0(K), R \subset
 R^t(s)\} \right),
\end{eqnarray*}
\end{theorem}
Figure~\ref{regions for concordance invariants}
describes the family of closed regions corresponding to each concordance invariant.
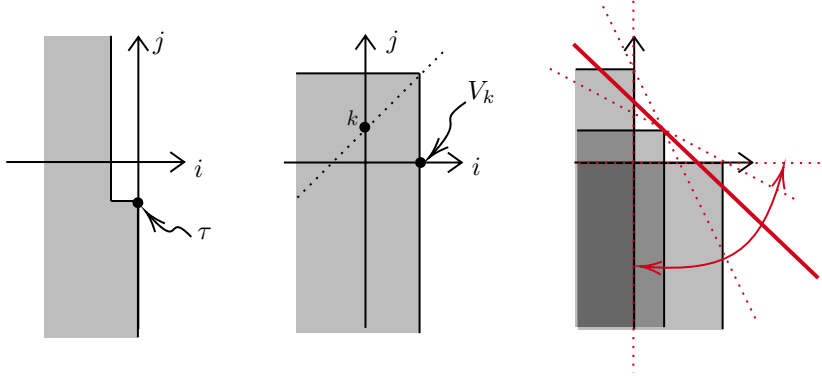
\begin{figure}[ht]
    \centering
    \tikzset{every picture/.style={line width=0.75pt}} 

\begin{tikzpicture}[x=0.75pt,y=0.75pt,yscale=-1,xscale=1]

\draw  (20.5,85.36) -- (110.21,85.36)(87.11,22.24) -- (87.11,169.7) (103.21,80.36) -- (110.21,85.36) -- (103.21,90.36) (82.11,29.24) -- (87.11,22.24) -- (92.11,29.24)  ;
\draw    (86.88,105.03) -- (86.88,173.82) ;

\draw  [draw opacity=0][fill={rgb, 255:red, 0; green, 0; blue, 0 }  ,fill opacity=0.26 ] (39.71,21.66) -- (73.31,21.66) -- (73.31,173.82) -- (39.71,173.82) -- cycle ;
\draw  [draw opacity=0][fill={rgb, 255:red, 0; green, 0; blue, 0 }  ,fill opacity=0.26 ] (73.31,105.03) -- (86.88,105.03) -- (86.88,173.82) -- (73.31,173.82) -- cycle ;
\draw    (73.31,21.66) -- (73.31,105.03) ;

\draw    (73.31,105.03) -- (86.88,105.03) ;

\draw  (160.5,85.61) -- (250.2,85.61)(201.65,21.61) -- (201.65,169.08) (243.2,80.61) -- (250.2,85.61) -- (243.2,90.61) (196.65,28.61) -- (201.65,21.61) -- (206.65,28.61)  ;
\draw    (228.94,40.61) -- (228.94,171.73) ;

\draw  [draw opacity=0][fill={rgb, 255:red, 0; green, 0; blue, 0 }  ,fill opacity=0.26 ] (166.67,40.61) -- (228.94,40.61) -- (228.94,171.73) -- (166.67,171.73) -- cycle ;
\draw    (166.67,40.61) -- (228.94,40.61) ;

\draw  [dash pattern={on 0.84pt off 2.51pt}]  (240.75,29.89) -- (166.33,104.31) ;

\draw  [fill={rgb, 255:red, 0; green, 0; blue, 0 }  ,fill opacity=1 ] (199.08,67.79) .. controls (199.08,66.56) and (200.08,65.56) .. (201.31,65.56) .. controls (202.54,65.56) and (203.54,66.56) .. (203.54,67.79) .. controls (203.54,69.02) and (202.54,70.02) .. (201.31,70.02) .. controls (200.08,70.02) and (199.08,69.02) .. (199.08,67.79) -- cycle ;
\draw  [fill={rgb, 255:red, 0; green, 0; blue, 0 }  ,fill opacity=1 ] (227.2,85.62) .. controls (227.2,84.39) and (228.2,83.39) .. (229.43,83.39) .. controls (230.66,83.39) and (231.66,84.39) .. (231.66,85.62) .. controls (231.66,86.85) and (230.66,87.85) .. (229.43,87.85) .. controls (228.2,87.85) and (227.2,86.85) .. (227.2,85.62) -- cycle ;
\draw  [fill={rgb, 255:red, 0; green, 0; blue, 0 }  ,fill opacity=1 ] (84.66,105.88) .. controls (84.66,104.65) and (85.65,103.65) .. (86.88,103.65) .. controls (88.12,103.65) and (89.11,104.65) .. (89.11,105.88) .. controls (89.11,107.11) and (88.12,108.11) .. (86.88,108.11) .. controls (85.65,108.11) and (84.66,107.11) .. (84.66,105.88) -- cycle ;
\draw    (113.78,123.18) .. controls (100.34,109.06) and (110.09,131.79) .. (91.95,111.62) ;
\draw [shift={(90.8,110.33)}, rotate = 408.69] [color={rgb, 255:red, 0; green, 0; blue, 0 }  ][line width=0.75]    (10.93,-3.29) .. controls (6.95,-1.4) and (3.31,-0.3) .. (0,0) .. controls (3.31,0.3) and (6.95,1.4) .. (10.93,3.29)   ;

\draw    (252.29,55.09) .. controls (235.05,67.2) and (254.68,61) .. (234.7,82.07) ;
\draw [shift={(233.43,83.39)}, rotate = 314.24] [color={rgb, 255:red, 0; green, 0; blue, 0 }  ][line width=0.75]    (10.93,-3.29) .. controls (6.95,-1.4) and (3.31,-0.3) .. (0,0) .. controls (3.31,0.3) and (6.95,1.4) .. (10.93,3.29)   ;

\draw  (307.01,85.61) -- (396.72,85.61)(337.19,22.3) -- (337.19,169.77) (389.72,80.61) -- (396.72,85.61) -- (389.72,90.61) (332.19,29.3) -- (337.19,22.3) -- (342.19,29.3)  ;
\draw    (337.05,38.56) -- (337.05,168.73) ;

\draw  [draw opacity=0][fill={rgb, 255:red, 0; green, 0; blue, 0 }  ,fill opacity=0.26 ] (307.7,38.56) -- (337.05,38.56) -- (337.05,168.73) -- (307.7,168.73) -- cycle ;
\draw    (307.7,38.56) -- (337.05,38.56) ;

\draw    (352.28,69.42) -- (352.28,168.73) ;

\draw  [draw opacity=0][fill={rgb, 255:red, 0; green, 0; blue, 0 }  ,fill opacity=0.26 ] (308.38,69.42) -- (352.28,69.42) -- (352.28,168.73) -- (308.38,168.73) -- cycle ;
\draw    (308.38,69.42) -- (352.28,69.42) ;

\draw    (381.78,85.88) -- (381.78,169.73) ;

\draw  [draw opacity=0][fill={rgb, 255:red, 0; green, 0; blue, 0 }  ,fill opacity=0.26 ] (309.07,85.88) -- (381.78,85.88) -- (381.78,169.73) -- (309.07,169.73) -- cycle ;
\draw    (309.07,85.88) -- (381.78,85.88) ;

\draw [color={rgb, 255:red, 208; green, 2; blue, 27 }  ,draw opacity=1 ][line width=1.5]    (306.56,25.81) -- (430.01,143.95) ;

\draw [color={rgb, 255:red, 208; green, 2; blue, 27 }  ,draw opacity=1 ] [dash pattern={on 0.84pt off 2.51pt}]  (321.42,6.73) -- (399.5,166.73) ;

\draw [color={rgb, 255:red, 208; green, 2; blue, 27 }  ,draw opacity=1 ] [dash pattern={on 0.84pt off 2.51pt}]  (294.67,38.29) -- (419.5,104.82) ;

\draw [color={rgb, 255:red, 208; green, 2; blue, 27 }  ,draw opacity=1 ] [dash pattern={on 0.84pt off 2.51pt}]  (309.07,85.88) -- (433.5,85.88) ;

\draw [color={rgb, 255:red, 208; green, 2; blue, 27 }  ,draw opacity=1 ] [dash pattern={on 0.84pt off 2.51pt}]  (336.78,3.97) -- (336.78,191.73) ;

\draw [color={rgb, 255:red, 208; green, 2; blue, 27 }  ,draw opacity=1 ][line width=0.75]    (342.64,138.21) .. controls (384.54,141.27) and (401.22,131.17) .. (412.01,89.65) ;
\draw [shift={(412.5,87.73)}, rotate = 464.04] [color={rgb, 255:red, 208; green, 2; blue, 27 }  ,draw opacity=1 ][line width=0.75]    (10.93,-3.29) .. controls (6.95,-1.4) and (3.31,-0.3) .. (0,0) .. controls (3.31,0.3) and (6.95,1.4) .. (10.93,3.29)   ;
\draw [shift={(340,138)}, rotate = 4.8] [color={rgb, 255:red, 208; green, 2; blue, 27 }  ,draw opacity=1 ][line width=0.75]    (10.93,-3.29) .. controls (6.95,-1.4) and (3.31,-0.3) .. (0,0) .. controls (3.31,0.3) and (6.95,1.4) .. (10.93,3.29)   ;

\draw (117.55,88.38) node   {$i$};
\draw (97.66,25.27) node   {$j$};
\draw (120.61,121.36) node   {$\tau $};
\draw (257.55,87.75) node   {$i$};
\draw (215.54,24.65) node   {$j$};
\draw (195.14,63.16) node [scale=0.8]  {$k$};
\draw (260.3,50.13) node   {$V_{k}$};

\end{tikzpicture}
    \caption{Closed regions corresponding to $\tau$, $V_k$, and  $\Upsilon$.}
    \label{regions for concordance invariants}
\end{figure}
Lastly, we give a formula for computing $V_k(K^*)$ from $\G_0(K)$, which is obtained via the duality theorem for $\G_0$.
\begin{proposition}
\label{prop: G_0 and mirror V_k}
 The invariant $V_k(K^*)$ is determined from $\G_0(K)$
 by the formula
 \[
 V_k(K^*) = \min\{m \in \z_{\geq 0} \mid 
 \forall R \in \G_0(K),\ (-m, -k-m) \in R\}.
 \]
\end{proposition}
\begin{proof}
For any $R \in \CR(\z^2)$, we see that 
\begin{eqnarray*}
(-m,-k-m) \in R
&\Leftrightarrow&
R \cap (-R_{(m,k+m)}) \neq \varnothing\\
&\Leftrightarrow&
R_{(m,k+m)} \cap (-R) \neq \varnothing.
\end{eqnarray*}
Therefore, it follows from \Cref{thm: duality}
that $R_{(m,k+m)}\in \tG_0(K^*)$ if and only if 
any $R \in \G_0(K)$ contains $(-m,-k-m)$.
Now, by \Cref{thm: minimalize} and \Cref{thm: G_0 and others}, we have
\begin{eqnarray*}
V_k(K^*)&=&\min\{m \in \z_{\geq 0} \mid  
\exists R \in \G_0(K^*), R \subset R_{(m,k+m)} \}\\
&=&\min\{m \in \z_{\geq 0} \mid  
R_{(m, k+m)} \in \tG_0(K^*) \}\\
&=&\min\{m \in \z_{\geq 0} \mid  
\forall R \in \G_0(K), (-m,-k-m) \in R \}.
\end{eqnarray*}
\end{proof}

\subsubsection{On the range of $\G_0$}
\label{range}
Here we discuss the range of $\G_0$. We first introduce 
several notions of closed regions.
For any subset $S \subset \z^2$, 
define the {\it closure} of $S$ by
$$
cl(S) := \bigcup_{(i,j) \in S} R_{(i,j)}.
$$
Then we also have $cl(S) \in \CR(\z^2)$. 
Moreover,  the equality 
$$
cl(S) = \bigcap_{R \in \CR(\z^2), S \subset R} R
$$ 
holds.
We say that a closed region $R \in \CR(\Z^2)$ is 
a {\it semi-simple region} if there exists
a non-empty finite subset $S \subset \Z^2$ such that $R=cl(S)$.

As examples of semi-simple regions, we define the {\it closure} of any chain $x = \sum_{1 \leq k \leq r} p_k(U) x_k \in CFK^\infty(K)$
by
\[
cl(x) := cl\left\{(\Alg(U^{l(p_k)}x_k), \Alex(U^{l(p_k)}x_k)) \  \middle| \  
\begin{array}{ll}
1 \leq k \leq r\\
p_k(U) \neq 0
\end{array}
\right\},
\]
where $l(p_k)$ denotes the lowest degree of 
$p_k(U)\in \Lambda$.
Note that the equality 
\[
cl(x) = \bigcap_{R \in \CR, x \in (CFK^\infty(K))_R} R\]
holds. (The proof is seen in  \cite[Lemma 5.5]{2019arXiv190709116S}.)

Let us denote the set of semi-simple regions by $\CR^{ss}(\Z^2)$.
Then, the following proposition immediately follows from 
\cite[Theorem 5.7]{2019arXiv190709116S}.
\begin{proposition}
\label{prop: realizer}
For any $R \in \G_0(K)$, there exists a homological generator
$x$ of $CFK^\infty(K)$ whose closure is equal to $R$.
In particular, for any knot $K$, we have 
\[
\G_0(K) \subset \CR^{ss}(\z^2).
\]
\end{proposition}
As a corollary, we have the following.
\begin{corollary}
\label{cor: realizer}
For any $R \in \G_0(K)$, there exists a non-empty finite 
subset $\{x_k\}_{k=1}^n \subset CFK^\infty(K)$
of linearly independent 0-chains
such that 
\begin{itemize}
    \item $\sum_{k=1}^n x_k$ is a homological generator of $CFK^\infty(K)$, and
    \item $R= cl(\{(\Alg(x_k), \Alex(x_k))\}_{k=1}^n)$.
\end{itemize}
\end{corollary}
Next, we introduce the notion of the {\it corners} of a semi-simple region, which will be used in the coming sections.

For a subset $S \subset \z^2$,
an element $s \in S$ is {\it maximal in $S$}
if it satisfies
$$
\text{if } s' \in S \text{ and } s' \geq s, \text{ then } s'=s.
$$
Define the map 
$$
\max \colon \mathcal{P}(\Z^2) \to \mathcal{P}(\Z^2)
$$
by
$$
S \mapsto \{ s \in S \mid s \text{ is maximal in } S\}.
$$ 
Then, for any $R \in \CR^{ss}(\z^2)$, the set of {\it corners} of $R$ is defined by
$$
c(R) := \max R.
$$
Here we prove that for $R \in \CR^{ss}(\z^2)$, the set
$c(R)$ is non-empty and finite. 
\begin{lemma}
\label{max S = max R}
For any $R \in \CR^{ss}(\z^2)$ and
non-empty finite set $S \subset \Z^2$ satisfying $cl(S)=R$,
the equality
$$
\max S = c(R)
$$
holds. In particular, $c(R)$ is non-empty and finite.
\end{lemma}

\begin{proof}
For proving Lemma~\ref{max S = max R}, we use the following lemma,
which is obtained as an analogy of 
\cite[Lemma~5.4]{2019arXiv190709116S}.
\begin{lemma}
\label{replace to maximal}
For any $s \in S$, there exists an element $s' \in \max S$ with $s' \geq s$. 
\end{lemma}
We first prove $\max S \subset \max R$.
 Let $s \in \max S$, and suppose that $p \in R$ satisfies $s \leq p$.
Then, since $cl(S)=R$, there exists an element $s' \in S$ such that $p \leq s'$.
Moreover, by Lemma~\ref{replace to maximal}, we have $s'' \in \max S$
with $s' \leq s''$. Now, we have 
$$
s \leq p \leq s' \leq s''.
$$
Here, since $s'' \in S$ and $s \in \max S$, we have $s = p = s' = s''$.
This implies $s \in \max R$.

Next, we prove $\max S \supset \max R$.
 Let $p \in \max R$, and suppose that $s \in S$ satisfies $p \leq s$.
Here, by Lemma~\ref{replace to maximal}, we may assume that $s \in \max S$.
Then, we have 
$$
s \in S \subset cl(S)=R,
$$
and hence the maximality of $p$ gives $p=s \in \max S$.
\end{proof}
Actually, the equality `$\max S = c(R)$' is a 
necessary and sufficient condition for a non-empty finite subset $S \subset \z^2$ to generate $R \in \CR^{ss}(\z^2)$.
Namely, we have the following.
\begin{lemma}
\label{lem: char ss}
For any semi-simple region $R$, 
a non-empty finite set $S \subset \Z^2$ satisfies $cl(S)=R$
if and only if $\max S = c(R)$. In particular,  $cl(c(R))=R$.
\end{lemma}

\begin{proof}
First, suppose that a non-empty finite set $S \subset \Z^2$ satisfies $cl(S)=R$. Then,
by Lemma~\ref{max S = max R}, we have
$$
\max S = \max R = c(R).
$$

Conversely, suppose that a non-empty finite set $S \subset \Z^2$ satisfies $\max S = c(R)$, and 
take a non-empty finite set $S' \subset \Z^2$ with $cl(S')=R$. Then we have
$$
\max S' = \max R = c(R)= \max S.
$$
Here, by Lemma~\ref{replace to maximal}, we see that
$$
cl(\max F) = cl(F)
$$
for any finite subset $F \subset \Z^2$. This implies
$$
cl(S) = cl(\max S) = cl(\max S') = cl(S') = R.
$$
\end{proof}

    \section{Translating $\G_0$ into $CFK^-$}
\label{sec: translate G_0}

In this section, we introduce shifted versions of 
$\G_0(K)$
which enable us to define $\G_0$-type invariants for
$CFK^-(K)$ and $C^-(G)$.
Moreover, we prove that
after shifting sufficiently large times, 
such shifted $\G_0$ for $CFK^-(K)$ recovers
the original $\G_0(K)$. 

\subsection{Shifted $\G_0$}
\label{generalization of G_0}
Let $C$ be a $\Z^2$-filtered chain complex over $\F$ such that 
$H_n(C) \cong \F$ for
given $n \in \Z$.
Then, a homogeneous cycle $x \in C_n$ is called a {\it homological generator of degree $n$} if the homology class $[x] \in H_n(C) \cong \F$ is non-zero.
Now we define 
$$
\tG_0^{(n)}(C) := \{R \in \CR(\Z^2) \mid 
\text{$C_R$ contains a homological generator of degree $n$}\}
$$
and
$$
\G_0^{(n)}(C) := \min \tG_0^{(n)}(C).
$$
Let $C$ and $C'$ be two $\Z^2$-filtered chain complexes over $\F$
with $H_n(C) \cong H_n(C') \cong \F$ for $n \in \Z$. Then, a chain map $f \colon C \to C'$ over $\F$ is a 
{\it $\Z^2$-filtered quasi-isomorphism at degree $n$}
if $f$ is $\Z^2$-filtered, graded and induces an isomorphism 
$f_{*,n} \colon H_n(C) \to H_n(C')$.
\begin{proposition}
\label{prop: G_0 deg n}
Let $C$ and $C'$ be two $\Z^2$-filtered chain complexes over $\F$
with $H_n(C) \cong H_n(C') \cong \F$ for given $n \in \Z$.
If there exists a $\Z^2$-filtered quasi-isomorphism 
$f \colon C \to C'$ at degree $n$, then we have
$$
\tG_0^{(n)}(C) \subset \tG_0^{(n)}(C').
$$
In particular, if $C \heq{\Z^2}{\F} C'$, then we have
$$
\tG_0^{(n)}(C) = \tG_0^{(n)}(C')
$$
and
$$
\G_0^{(n)}(C) = \G_0^{(n)}(C').
$$
\end{proposition}
\begin{proof}
If $x\in C_R$ is a homological generator of degree $n$, then
$f(x)\in C'$ is also a homological generator of degree $n$,
and $f(x) \in f(C_R) \subset C'_R$. 
\end{proof}
Here let us consider the case of $CFK^{\infty}$.
By the global triviality, 
we can define $\G^{(n)}_{0}(CFK^\infty(K))$
for any even integer $n$.
We show that all $\G^{(n)}_{0}(CFK^\infty(K))$ can recover one another. To show that, we define {\it shifts} of closed regions.

Let $R \in \CR(\Z^2)$ and $s \in \Z$. Then we define
$$
R[s] := \{(i,j) \in \z^2 \mid (i+s,j+s) \in R\}.
$$
For instance, a simple region $R_{(k,l)}$ is shifted to
$R_{(k-s,l-s)}$ by $s \in \z$, i.e.\ $(R_{(k,l)})[s]=R_{(k-s,l-s)}$.
Obviously, for any $s,t \in \z$, we have
$$
(R[s])[t]=R[s+t]
$$
and
$$
R[0]=R.
$$

Next, for $\mathcal{S} \subset \CR(\Z^2)$ and $s \in \z$,
we define
$$
\mathcal{S}[s] := \{R[s] \in \CR(\z^2) \mid R \in \mathcal{S} \},
$$
and then we have
$(\mathcal{S}[s])[t]= \mathcal{S}[s+t]$ and $\mathcal{S}[0]=\mathcal{S}$.

\begin{proposition}
\label{prop: shifted G_0 for CFK}
For any knot $K$ and $s \in \Z$, we have 
$$
\tG_0^{(-2s)}(CFK^\infty(K))= \tG_0(K)[s].
$$
and
$$
\G_0^{(-2s)}(CFK^\infty(K))= \G_0(K)[s].
$$
\end{proposition}

\begin{proof}
By the fourth property of $CFK^\infty$, the following hold:
\begin{itemize}
\item An element $x \in C$ is a homological generator
of degree $-2s$ if and only if
$U^{-s}x$ is a homological generator (of degree $0$).
\item The equality $U^{-s}\cdot C_R
= C_{R[-s]}$ holds.
\end{itemize}
In particular, $C_R$ contains a homological generator
of degree $-2s$ if and only if 
$C_{R[-s]}$ contains a homological generator (of degree 0).
This implies 
\begin{eqnarray*}
R \in \tG^{(-2s)}_0(C) &\Leftrightarrow& 
R[-s] \in \tG_0(C)\\
&\Leftrightarrow& 
R \in \tG_0(C)[s].
\end{eqnarray*}
This completes the proof.
\end{proof}

As corollaries of Proposition~\ref{prop: shifted G_0 for CFK},
we see that $\G_0^{(-2s)}(C)$ shares several nice properties
such as Theorem~\ref{thm: G'_0} and Theorem~\ref{thm: minimalize} with $\G_0(C)$.

\begin{corollary}
\label{cor: shifted G'_0}
For any $s \in \Z$, the set $\G^{(-2s)}_0(C)$ 
is non-empty and finite.
\end{corollary}

\begin{corollary}
\label{cor: shifted minimalize}
For any $s \in \Z$ and closed region $R$, the following holds:
$$
R \in \tG^{(-2s)}_0(C) \Leftrightarrow \exists R' \in 
\G^{(-2s)}_0(C), \  R' \subset R.
$$
\end{corollary}

    \subsection{Recovering $\G_0(K)$ from 
$\G_0^{(-2s)}(CFK^-(K))$}
Let $C := CFK^\infty(K)$.
Here we consider 
$\G^{(n)}_0(CFK^-(K))= \G^{(n)}_0(\mF^{\Alg}_0(C))$.
By the seventh property of $CFK^\infty$,
we have the isomorphism 
$$
H_{n}(\mF^{\Alg}_0(C))\cong H_n(\mF^{\Alg}_0(\Lambda))\cong 
\begin{cases}
\F & (n \leq 0 \text{ and $n \colon$ even})\\
0 & (\text{otherwise})
\end{cases}.
$$
Therefore, we can define 
$\tG_0^{(-2s)}(\mF^{\Alg}_0(C))$ and
$\G_0^{(-2s)}(\mF^{\Alg}_0(C))$
for any $s \in \Z_{\geq 0}$.
Here, $\mF^{\Alg}_0(C)$ is regarded as a $\Z^2$-filtered chain complex (over $\F[U]$) by
$$
(\mF^{\Alg}_0(C))_R := \mF^{\Alg}_0(C) \cap C_R
$$
for each $R \in \CR(\Z^2)$.

In order to state the relation between $\G_0^{(-2s)}(\mF^{\Alg}_0(C))$ and
$\G_0(C)$, we first introduce the {\it shift number} of 
a closed region. 
For any $R \in \CR(\Z^2)$, we define the \textit{shift number} of $R$ by 
\[
    \shift(R) := \max\{ i \in \Z \mid \exists j \in \z \ \text{s.t.}\ (i, j) \in R \}.
\]
Here we define $\shift(\varnothing) := -\infty$
and $\shift(R):= \infty$ if there exists no integer $M$ such that $R \subset \{i\leq M\}$.
Note that $\shift(R)$ is finite if $R \in \CR^{ss}(\Z^2)$.
In addition, if $R \subset R'$, then $\shift(R) \leq \shift(R')$.

Now, we can state the relation between
$\G^{(-2s)}_0(CFK^-(K))$ and $\G_0(K)$ as follows.
\begin{theorem}
\label{thm: shifted G_0 for CFK^-}
For any knot $K$ and
$s \in \Z_{\geq 0}$, we have 
$$
\G_0^{(-2s)}(CFK^-(K))= 
\{R[s] \mid R \in \G_0(K), \  \shift(R)\leq s \}.
$$
\end{theorem}
To prove \Cref{thm: shifted G_0 for CFK^-},
we use the following three lemmas.
\begin{lemma}
\label{lem: CFK^- 1}
For any $R \in \CR(\Z^2)$, we have
$$
(\mF^{\Alg}_0(C))_R = C_{R \cap \{i \leq 0\}}
= (\mF^{\Alg}_0(C))_{R \cap \{i \leq 0 \}}.
$$
\end{lemma}

\begin{proof}
By \cite[Lemma 2.9]{2019arXiv190709116S}, there exists 
a direct decomposition 
\[
C = \bigoplus_{(i,j) \in \Z} C_{(i,j)}
\]
as a $\F$-vector space
such that for any $R \in \CR(\Z^2)$, the equality
$$
C_R = \bigoplus_{(i,j) \in R} C_{(i,j)}
$$
holds.
Now, since $\mF^{\Alg}_0(C)=C_{\{i \leq 0\}}$, it is easy to see that
\begin{eqnarray*}
(\mF^{\Alg}_0(C))_R &=& C_{\{i \leq 0\} } \cap C_R\\
&=&  (\bigoplus_{i \leq 0}C_{(i,j)} ) \cap 
(\bigoplus_{(i,j)\in R}C_{(i,j)} )\\
&=&
\bigoplus_{(i,j)\in R \cap \{i \leq 0\}}C_{(i,j)}
= C_{R \cap \{ i\leq 0\}}.
\end{eqnarray*}
Now, the second equality follows from
$
C_{R \cap \{i \leq 0\}} = (\mF^{\Alg}_0(C))_R \subset \mF^{\Alg}_0(C). 
$
\end{proof}

\begin{lemma}
\label{lem: CFK^- 2}
If $R \in \tG^{(-2s)}_0(\mF^{\Alg}_0(C))$,
then 
there exists an element
$R' \in \G^{(-2s)}_0(C)$ such that
$R' \subset R$ and $\shift(R') \leq 0$.
\end{lemma}

\begin{proof}
Since 
$(C, \{\mF^{\Alg}_i(C)\}) \heq{\Z}{\F} 
(\Lambda, \{\mF^{\Alg}_i(\Lambda)\})$, 
the inclusion $i \colon \mF^{\Alg}_0(C) \to C$
induces the isomorphism $H_{-2s}(\mF^{\Alg}_0(C)) \cong H_{-2s}(C)$. Moreover, by the definition of the $\Z^2$-filtration on $\mF^{\Alg}_0(C)$, it is obvious that
the map $i$ is $\Z^2$-filtered.
Therefore, the inclusion $i \colon \mF^{\Alg}_0(C) \to C$ 
is a $\Z^2$-filtered quasi-isomorphism at degree $-2s$, 
and hence Proposition~\ref{prop: G_0 deg n} gives
$$
\tG^{(-2s)}_0(\mF^{\Alg}_0(C)) \subset 
\tG^{(-2s)}_0(C).
$$
Now, let $R \in \tG^{(-2s)}_0(\mF^{\Alg}_0(C))$.
Then, Lemma~\ref{lem: CFK^- 1} and the above arguments imply that
$$
R \cap \{i \leq 0\} \in \tG^{(-2s)}_0(\mF^1_0(C))
\subset \tG^{(-2s)}_0(C).
$$
Therefore, by Corollary~\ref{cor: shifted minimalize}, 
there exists an element $R' \in \G^{(-2s)}_0(C)$
such that $R' \subset R \cap \{i \leq 0\}$.
Obviously, we have $R' \subset R$ and
$$
\shift(R') \leq \shift(R \cap \{i \leq 0\}) \leq 0.  
$$
\end{proof}

\begin{lemma}
\label{lem: CFK^- 3}
If $R \in \G^{(-2s)}_0(C)$ and $\shift(R) \leq 0$,
then $R \in \tG^{(-2s)}_0(\mF^{\Alg}_0(C))$.
\end{lemma}

\begin{proof}
Note that for any $R \in \CR(\Z^2)$, 
the inequality $\shift(R) \leq 0$ holds
if and only if $R \subset \{i \leq 0\}$.
Now, let $R \in \G^{(-2s)}_0(C)$ and $\shift(R) \leq 0$.
Then there exists
a homological generator $x \in C$ of degree $-2s$
which satisfies
$$
x \in C_{R} \subset C_{\{i \leq 0\}} = \mF^{\Alg}_0(C).
$$
In particular, 
$x \in \mF^{\Alg}_0(C) \cap C_R = (\mF^{\Alg}_0(C))_R$,
and hence we can regard $x$ as a homological generator 
of degree $-2s$ for $\mF^{\Alg}_0(C)$ which lies in $(\mF^{\Alg}_0(C))_R$.
This shows $R \in \tG^{(-2s)}_0(\mF^{\Alg}_0(C))$.
\end{proof}
Now, let us prove \Cref{thm: shifted G_0 for CFK^-}.
\def\proofname{Proof of \Cref{thm: shifted G_0 for CFK^-}}
\begin{proof}
For any $R \in \CR^{ss}(\Z^2)$, we see that
$$
\shift(R) \leq s \Leftrightarrow \shift(R[s]) \leq 0.
$$
Hence, Proposition~\ref{prop: shifted G_0 for CFK}
implies
\begin{eqnarray*}
\{R[s] \mid R\in \G_0(K), \ \shift(R) \leq s\}
&=&
\{R \in \G_0(K)[s] \mid  \shift(R) \leq 0\}\\
&=&
\{R \in \G^{(-2s)}_0(C) \mid  \shift(R) \leq 0\}.
\end{eqnarray*}
Therefore, it suffices to prove that
$$
\G_0^{(-2s)}(\mF^{\Alg}_0(C))
=\{R \in \G^{(-2s)}_0(C) \mid  \shift(R) \leq 0\}.
$$
We first suppose that 
$R \in \G_0^{(-2s)}(\mF^{\Alg}_0(C))$.
Then $R \in \tG_0^{(-2s)}(\mF^{\Alg}_0(C))$,
and hence Lemma~\ref{lem: CFK^- 2} gives an element
$R' \in \G_0^{(-2s)}(C)$ such that
$R' \subset R$ and $\shift(R') \leq 0$.
Moreover, Lemma~\ref{lem: CFK^- 3} shows that
$R'\in \tG_0^{(-2s)}(\mF^{\Alg}_0(C))$.
By the minimality of $R$ in
$\tG_0^{(-2s)}(\mF^{\Alg}_0(C))$,
we must have $R=R' \in \G_0^{(-2s)}(C)$.

Next, we suppose that 
$R \in \G_0^{(-2s)}(C)$ and 
$\shift(R) \leq 0$.
Then, Lemma~\ref{lem: CFK^- 3}
shows that 
$R \in \tG_0^{(-2s)}(\mF^{\Alg}_0(C))$.
Assume that $R' \in \tG_0^{(-2s)}(\mF^{\Alg}_0(C))$ 
satisfies $R' \subset R$.
Then, 
Lemma~\ref{lem: CFK^- 2} gives an element
$R'' \in \G_0^{(-2s)}(C) \subset \tG_0^{(-2s)}(C)$ such that
$R'' \subset R' \subset R$.
By the minimality of $R$ in $\tG_0^{(-2s)}(C)$,
we must have $R''=R'=R$. This implies that 
$R \in \G_0^{(-2s)}(\mF^{\Alg}_0(C))$.
\end{proof}
\def\proofname{Proof}

As a corollary of \Cref{thm: shifted G_0 for CFK^-}, we see that the sequence
\[\{\G_0^{(-2s)}(CFK^-(K))\}_{s=0}^\infty
\]
converges
to shifted $\G_0(K)$.
\begin{corollary}
For any $s \geq \max\{\shift (R) \mid R \in \G_0(K)\}$,
we have
\[
\G_0^{(-2s)}(CFK^-(K))= \G_0(K)[s].
\]
\end{corollary}

    \section{Translating $\G_0$ into $C^-(G)$} 
\label{sec: translating G_0 to C^-(G)}
In this section, we recall the grid complex $C^-(G)$ and its relationship to $CFK^-$. 
Combining such arguments with \Cref{thm: shifted G_0 for CFK^-}, 
we prove that 
$\G_0(K)$ is determined from $\G^{(-2s)}_0(C^-(G))$
for sufficiently large $s \in \Z_{\geq 0}$.

\subsection{A review of grid complexes} 
\label{subsec:grid-cpx}

\textit{Grid homology theory} is a combinatorial description of knot Floer homology theory, introduced by Manolescu, Ozsv\'ath, Sarkar, Szab\'o and Thurston in \cite{MOS09grid1, MOST07grid2}. For any knot $K \subset S^3$, by choosing a \textit{grid diagram} $G$ of $K$ (also called an \textit{arc representation}), one associates the \textit{grid complex} $C^-(G)$ which is
filtered homotopy equivalent to $CFK^-(K)$. Here we briefly review the construction.

\begin{figure}[t]
    \centering
    \input{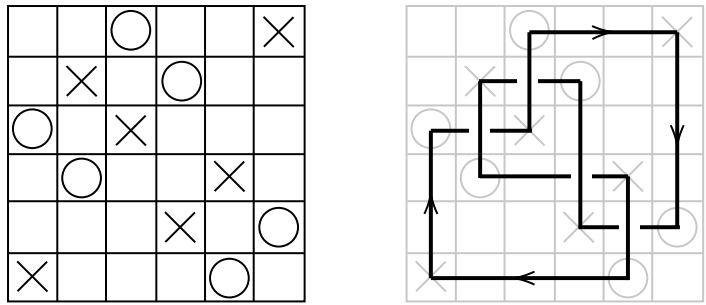}
    \caption{Grid diagram and the corresponding knot}
\end{figure}

A \textit{grid diagram} $G$ lies on an $n \times n$ grid of squares, where some squares are decorated either with an $O$ or an $X$ so that
\begin{itemize}
    \item every row contains exactly one $O$ and one $X$;
    \item every column contains exactly one $O$ and one $X$.
\end{itemize}
The number $n$ is called the \textit{grid number} of $G$. Given such data, one obtains a planar link diagram by drawing horizontal segments from the $O$'s to the $X$'s in each row, and vertical segments from the $X$'s to the $O$'s in each column, while letting the horizontal segment underpass the vertical segment at every intersection point. It is also true that every link in $S^3$ possesses such representation. We usually place the diagram in the standard plane so that bottom left corner is at the origin, each square has length one, and each $O$ and each $X$ is centered at a half-integer point. We set $\mathbb{O} = \{O_i\}_{1 \leq i \leq n},\ \mathbb{X} = \{X_i\}_{1 \leq i \leq n}$ so that each $O_i$ and $X_i$ has its center in $x = i - \frac{1}{2}$. 

Given a grid diagram $G$, the chain complex $C^-(G)$ is constructed as follows. First we regard $G$ as a diagram on the torus by gluing the two opposite sides. The generating set $S$ is given by $n$-tuples of intersection points between the horizontal and vertical circles, with the property that each horizontal (or vertical) circle contains exactly one intersection point. By fixing the bottom left corner of the grid, each $\x \in S$ can be identified with the set of permutations of length $n$ under the correspondence
\[
    \sigma \mapsto \x = \{ (i, \sigma(i)) \}_{0 \leq i < n }.
\]

$C^-(G)$ is generated by $S$ over the multivariate polynomial ring $\F[U_1, \cdots, U_n]$. The differential $\partial$ is defined as 
\[
    \partial(\x) = \sum_{\y \in S}\sum_{r \in \mathrm{Rect}^\circ (\x, \y)} U_1^{\epsilon_1} \cdots U_n^{\epsilon_n} \mathbf{y}
\]
where $\mathrm{Rect}^\circ(\x, \y)$ denotes the set of \textit{empty rectangles} connecting $\x$ to $\y$ (which exist only when $\x$ and $\y$ are related by a single transposition), and for each empty rectangle $r$, the exponent $\epsilon_i \in \{0, 1\}$ is given by the number of intersections of $r$ and $O_i$. \Cref{fig:grid_diffential} depicts an empty rectangle $r$ connecting $\x$ to $\y$, and for this $r$ we have $\epsilon_4 = 1$. See \cite{MOST07grid2} for the precise definition.

\begin{figure}[t]
    \centering
    \input{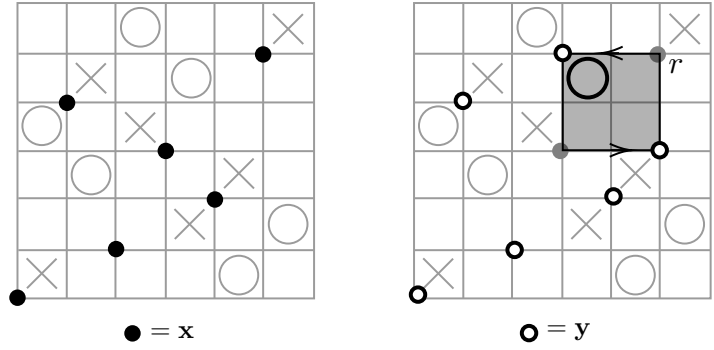}
    \caption{An empty rectangle connecting $\x$ to $\y$}
    \label{fig:grid_diffential}
\end{figure}

$C^-(G)$ is endowed the \textit{Maslov grading} and the \textit{Alexander grading} as follows. If $P$ and $Q$ are sets of finitely many points in $\R^2$, we define $I(P, Q)$ by the number of pairs $p \in P$ and $q \in Q$ with $p < q$. We symmetrize this function as
\[
    J(P, Q) := I(P, Q) + I(Q, P).
\]
Here $J$ is extended bilinearly over formal sums of finite subsets of $\R^2$. For any $\x \in S$, we define
\begin{align*}
    M(\x) &:= J(\x - \mathbb{O}, \x - \mathbb{O}) + 1, \\
    A(\x) &:= J(\x - \frac{1}{2}(\mathbb{X} + \mathbb{O}), \mathbb{X} - \mathbb{O}) - \frac{n - 1}{2}.
\end{align*}{}
We also declare that each factor $U_i$ decreases $M$ by $-2$ and $A$ by $-1$. Then it follows that $\partial$ decreases $M$ by $-1$, while $A$ is non-increasing under $\partial$. Thus $M$ gives a homological grading of $C^-(G)$ and $A$ gives a $\z$-filtration $\{\falex{j}\}$ on $C^-(G)$. We regard $C^-(G)$ as a $\Z$-filtered complex over $\F[U]$, where the action by $U$ is defined to be multiplication by $U_1$.

\begin{theorem}[{\cite[Theorem 3.3]{MOS09grid1}}]
\label{thm: Z-filtered C^-(G)}
    If $G$ is a grid diagram of a knot $K$, then
    we have
\[(C^-(G),\falex{j}) 
\heq{\z}{\F[U]} (CFK^-(K),\falex{j}).
\]
\end{theorem}

\subsection{Computing $\G_0(K)$ from $C^-(G)$} \label{subsec:G_0-from-grid-cpx}

Here we show that the $\z$-filtered homotopy equivalence
stated in \Cref{thm: Z-filtered C^-(G)}
induces the $\z^2$-filtered homotopy equivalence between $C^-(G)$
and $CFK^-(K)$.
To show that, we introduce
the {\it algebraic filtration}
$\{\falg{i}\}$ on $C^-(G)$
as follows:
For any $i \in \z$, the $i$-th subcomplex $\falg{i}$
is defined to be the $\F[U]$-submodule generated by elements of the form 
$U_1^{a_1} \cdots U_N^{a_N} \x
\in C^-(G)$
with $a_1 \geq -i$.
By the definition of the differential on $C^-(G)$,
it is obvious that $\{\falg{i}\}$ is 
a $\z$-filtration on $C^-(G)$.
Moreover, we see that 
\[
U \cdot \falg{i} = \falg{i-1}
\]
for any $i \leq 0$,
and 
\[
\falg{i} = C^-(G)
\]
for any $i \geq 0$. (Note that $CFK^-(K)$ also has the same property.)
We use the $\z^2$-filtration
induced by $(\{\falg{i}\},\{\falex{j}\})$
to regard $C^-(G)$
as a $\z^2$-filtered complex
over $\F[U]$. 
Then, we have the following.
\begin{corollary}
\label{cor: Z^2-filtered C^-(G)}
    If $G$ is a grid diagram of a knot $K$, then
    we have
\[
C^-(G) 
\heq{\z^2}{\F[U]} CFK^-(K).
\]
\end{corollary}

\begin{proof}
Let $C$ and $C'$ be either one of
$C^-(G)$ and $CFK^-(K)$, 
and 
suppose that $f \colon C \to C'$
is a map over $\F[U]$. 
Then, obviously we have
\[
f(\falg{i}(C)) 
\subset C' = \falg{i}(C')
\]
if $i \geq 0$.
Moreover, if $i < 0$, then
\begin{eqnarray*}
f(\falg{i}(C)) &=& 
f(U^{-i} \cdot\falg{0}(C)) \\
&=&
U^{-i} \cdot f(\falg{0}(C)) \subset
U^{-i} \cdot \falg{0}(C') = \falg{i}(C').
\end{eqnarray*}
Now, it is easy to see that a $\z$-filtered homotopy equivalence map given by 
\Cref{thm: Z-filtered C^-(G)} 
induces $C^-(G) \heq{\z^2}{\F[U]} CFK^-(K)$.
\end{proof}


Now, combining \Cref{cor: Z^2-filtered C^-(G)}
with \Cref{prop: G_0 deg n} and \Cref{thm: shifted G_0 for CFK^-}, we have the following.
\begin{theorem}
\label{thm: G_0 and C^-(G)}
For any knot $K$, grid diagram $G$ for $K$ and $s \in \Z_{\geq 0}$,
we have
$$
\G_0^{(-2s)}(C^-(G)) = \{R[s] \mid R \in \G_0(K), \  \shift(R) \leq s\}. 
$$
\end{theorem}
We also have the following convergence theorem.
\begin{corollary}
For any $s \geq \max\{\shift (R) \mid R \in \G_0(K)\}$,
we have
\[
\G_0^{(-2s)}(C^-(G))= \G_0(K)[s].
\]
\end{corollary}
Moreover, we have the following corollary,
which enables us to check the realizability of each closed region with finite shift number via $C^-(G)$.
\begin{corollary}
\label{cor: tG_0 and C^-(G)}
For any closed region $R$ with $\shift(R) \leq s$,
we have
$$
R \in \tG_0(K) \Leftrightarrow
R[s]\in \tG_0^{(-2s)}(C^-(G)).
$$
\end{corollary}

\begin{proof}
First, suppose that $R \in \tG_0(K)$.
Then, by \Cref{thm: minimalize}, we have 
$R' \in \G_0(K)$ with $R' \subset R$.
In particular, 
the inequalities $\shift(R') \leq \shift (R) \leq s$ hold,
and hence \Cref{thm: G_0 and C^-(G)}
gives $R'[s] \in \G_0^{(-2s)}(C^-(G))$.
This implies that $R[s]\in \tG_0^{(-2s)}(C^-(G))$.

Next, suppose that $R[s] \in \tG_0^{(-2s)}(C^-(G))$.
Then \Cref{cor: Z^2-filtered C^-(G)} gives
$R[s] \in \tG_0^{(-2s)}(CFK^-(K))$.
Moreover, \Cref{lem: CFK^- 2} 
and \Cref{prop: shifted G_0 for CFK} implies
\[
R[s] \in \tG_0^{(-2s)}(CFK^\infty(K))= \tG_0(K)[s].
\]
Therefore, we have $R \in \tG_0(K)$.
\end{proof}
    \section{Algorithm}
\label{sec: algorithm}

Now we describe the algorithm to compute $\mathcal{G}_0(K)$. First we describe the four main procedures which assures that the algorithm exists, and then describe how these procedures are integrated so that the computations are done effectively. Throughout this section we assume that we are given a grid diagram $G = (\mathbb{O}, \mathbb{X})$ of $K$, and denote $C = C^-(G)$.

\subsection{Enumerating the candidate regions}

First we show that there exists a finite set of closed regions that is assured to include $\G_0(K)$. 

\begin{proposition} \label{prop:G_0-torus-knot}
    For any $g \geq 0$,
    \begin{align*}
        \G_0(T_{2, 2g + 1})  &= \{ R_{(i, g - i)} \mid i = 0, \cdots, g \}, \\
        \G_0((T_{2, 2g + 1})^*) &= \{ R_{(-g, 0)} \cup R_{(-g + 1, -1)} \cup \cdots \cup R_{(0, -g)} \}
    \end{align*}
\end{proposition}

\begin{proof}
    It is known that $CFK^\infty((T_{2, 2r + 1})^*)$ possesses a unique homological generator $a := a_0 + \cdots a_g$ such that $(\Alg(a_i), \Alex(a_i)) = (-g + i, -i)$. This implies the latter equality. 
    Similarly, we can easily compute $\G_0(T_{2,2g+1})$.
    (We can also determine $\G_0(T_{2,2g+1})$ from $\G_0((T_{2,2g+1})^*)$ using \Cref{char G_0 of mirror}.)
\end{proof}

\begin{proposition} \label{prop:candidate-cond}
    Let $g_3, g_4$ denote the 3-, 4- genus of a knot $K$ respectively. Any $R \in \G_0(K)$ satisfies the following three conditions:
    \begin{enumerate}
    \item 
        Each corner $(i, j)$ of $R$ satisfies $|i - j| \leq g_3$.
    \item
        $R$ includes the region $R_{(-g_4, 0)} \cup R_{(-g_4 + 1, -1)} \cup \cdots \cup R_{(0, -g_4)}$.
    \item
        If $R$ contains a point $(i, g_4 - i)$ for some $i \in \{0, \ldots, g_4 \}$, then $R = R_{(i, g_4 - i)}$.
    \end{enumerate}{}
\end{proposition}

\begin{proof}
    By \cite[Theorem 4.5]{2019arXiv190709116S}, up to $\z^2$-filtered homotopy equivalence, we may assume that
    $CFK^\infty$ has a filtered basis $\{x_k\}_{1 \leq k \leq r}$ such that
    \[
    |\Alex(U^n x_k) - \Alg(U^n x_k)| \leq g_3
    \]
    for any $1 \leq k \leq r$ and $n \in \z$.
    Hence the first statement follows from \Cref{prop: realizer}. Next, from \cite[Section 3.1]{Pet13}, we have $\pm g[T_{2,3}]_{\nu^+} = \pm[T_{2,2g + 1}]_{\nu^+}$. Now the remaining two statements follow from 
    \Cref{thm:g4-bound}, \Cref{prop:nuplus-ineq} and \Cref{prop:G_0-torus-knot}.
\end{proof}

In particular, any corner $(i, j)$ of $R \in \G_0(K)$ lies in the bounded area
\begin{align*}
    |i + j| \leq g_4 &\quad ( ij \geq 0), \\
    |i - j| \leq g_3 &\quad ( ij < 0).
\end{align*}
Thus the set of all semi-simple regions satisfying the conditions of \Cref{prop:candidate-cond} is finite. Furthermore, another strong condition can be imposed from the bigraded module structure of $\widehat{HFK}(K)$. As mentioned in \cite[Section 4.]{KP18}, by \cite[Lemma 4.5]{rasmussen2003floer}, up to $\Z^2$-filtered homotopy equivalence, $CFK^\infty(K)$ may be regarded as a chain complex generated by $\widehat{HFK}(K)$ over $\F[U, U^{-1}]$. Define
\[
    \widehat{S}_k := \{\ U^a x \mid M(x) \equiv k \bmod{2},\ a = (M(x) - k) / 2 \ \}
\]
where $x$ runs over the homogeneous generators of $\widehat{HFK}(K)$. For a closed region $R$, we say two corners $(i, j)$ and $(k, l)$ of $R$ are \textit{adjacent} if there are no corners in between them.

\begin{proposition} \label{prop:HFK-hat-reduction}
    Any $R \in \G_0(K)$ satisfies the following two conditions:
    \begin{enumerate}
    \item 
        For each corner $(i, j)$ of $R$, there exists an element $x \in \widehat{S}_0$ such that $(\Alg(x), \Alex(x)) = (i, j)$.
    \item
        For each pairs of adjacent corners $(i, j)$ and $(k, l)$ of $R$ with $i < k$, there exists an element $x \in \widehat{S}_{-1}$ such that $(\Alg(x), \Alex(x)) \leq (i, l)$.
    \end{enumerate}
\end{proposition}

\begin{proof}
    The first statement is obvious from \Cref{cor: realizer}
    and the above arguments. For the second statement, take a homological generator $z$ of degree 0 whose closure gives $R$. Decompose $z$ as $x + y$ so that $x$ consists of terms of $z$ having $\Alg \leq i$. Obviously $x, y \neq 0$ and $\partial x + \partial y = 0$. Also from the minimality of $R$, we have $\partial x \neq 0$. Thus $\partial x = \partial y$ belongs to $cl(x) \cap cl(y) = R_{(i, l)}$.
\end{proof}{}

If $g_3(K), g_4(K)$ and $\widehat{HFK}(K)$ are known beforehand, we take them as additional inputs. If not, then we can compute $\widehat{HFK}(K)$ (which also gives $g_3(K)$) from the simplified grid complex
\[
    \widetilde{C}(G) := \frac{C^-(G)}{U_1 = \cdots = U_n = 0}.
\]
From \cite[Theorem 3.6.]{MOST07grid2} we have
\[
    H(\widetilde{C}(G)) \cong \widehat{HFK}(K) \otimes W^{\otimes (n - 1)},
\]
where $W$ is the two-dimensional bigraded vector space spanned by two generators, one generator in bigrading $(0, 0)$ and another in bigrading $(-1, -1)$. An effective algorithm for computing $\widehat{HFK}$ is described in \cite{BG12}. 

We call the set of all semi-simple regions satisfying \Cref{prop:candidate-cond} and \ref{prop:HFK-hat-reduction} the set of \textit{candidate regions}. Having enumerated the candidate regions, the remaining task is to determine the realizability for each candidate region. In fact, from the symmetry $CFK^{\infty}(K) \heq{\z^2}{\Lambda} (CFK^{\infty}(K))^r$, it follows that
\[
    R \in \G_0(K) \Leftrightarrow R^r \in \G_0(K).
\]
where $R^r$ is the \textit{reflection} of $R$
\[
    R^r := \{ (i, j) \mid (j, i) \in R \}.
\]
Thus it suffices to consider only one of the reflection pairs. For reasons to be explained in the coming sections, we check the one that has the smaller shift number.
    \subsection{Inflating the generators}\label{subsec:gens}

Recall that the grid complex $C = C^-(G)$ is generated by the set $S$ over $\F[U_1, \cdots, U_n]$, where each $\x \in S$ corresponds one-to-one to a permutation of length $n$. \textit{Heap's algorithm} \cite{heap1963permutations} is well known in the area of computer science, which enumerates all permutations of any fixed length by a sequence of transpositions. As for the two gradings $M$ and $A$, if $\x, \y \in S$ are related by a transposition and $r$ is a rectangle that connects $\x$ to $\y$, it is known that the following relations hold:
\begin{align*}
    M(\y) - M(\x) &= 2 \#(r \cap \mathbb{O}) - 2 \#(\x \cap \mathrm{Int}(r)) - 1, \\
    A(\y) - A(\x) &= \#(r \cap \mathbb{O}) - \#(r \cap \mathbb{X}).
\end{align*}
Note that computing the two gradings sequentially from the above formulas is more effective than computing them directly using the formulas given in \Cref{subsec:grid-cpx}. Thus we obtain the generating set $S$ together with the gradings $(M(\x), A(\x))$ for each $\x \in S$. Next we inflate these generators by multiplying monomials in $U_1, \cdots U_n$ and regard $C$ as an (infinitely generated) chain complex over $\F$. Each factor $U_i$ decreases the homological degree by $2$, so for any $k \in \Z$, the $k$-th chain group $C_k$ can be regarded as a finite dimensional vector space over $\F$ with generators of the form
$$ 
U_1^{a_1} \cdots U_N^{a_N} \x, \quad \deg{\x} - 2\textstyle{\sum_i} a_i = k.
$$
We call generators of this form \textit{inflated generators}. Recall from \Cref{subsec:grid-cpx} that the $\Z^2$-filtration is given by
\begin{align*}
    \Alg(U_1^{a_1} \cdots U_N^{a_N} \x) &= -a_1,\\
    \Alex(U_1^{a_1} \cdots U_N^{a_N} \x) &= A(\x) - \sum_i a_i.
\end{align*}


Note that the number of inflated generators increases infinitely as the homological degree $k$ decreases. 
Nonetheless, from \Cref{prop:candidate-cond} (and reasons to be stated in the coming sections), we only need to consider chain groups of degree between $-2 g_3$ and 1, hence the enumeration is finite.
    \subsection{Computing a homological generator}

Next we describe the algorithm for computing a homological generator of $C(G)$, i.e.\ a representative cycle of the unique generator of $H^-_0(G) \cong \F$. For efficiency, the algorithm avoids direct computation of the homology group. We call the set of points $\{ (i, n - i - 1) \mid i \in \frac{1}{2}\Z \}$ the \textit{anti-diagonal} of $G$.

\begin{proposition}
    Suppose all $O$'s of $G$ lie in the anti-diagonal of $G$. The element $\x_0 := \{ (i, n - i - 1) \}_{0 \leq i < n}$ represents the unique generator of $H^-_0(G) \cong \F$.
\end{proposition}

\begin{proof}
    $M(\x_0) = 0$ is proved in \cite[Lemma 4.3.5]{OSS15gridbook}. For any $\y$ that is related to $\x_0$ by a transposition, there are exactly two empty rectangles that connect $\x_0$ to $\y$, and neither of them contains $O$. Thus $\partial \x_0 = 0$. Conversely, rectangles that connect $\y$ to $\x_0$ contains either a point of $\x_0$ or one of $O$, so there is no chain $c$ such that $\partial c$ contains $\x_0$ with coefficient 1 in its terms. Thus $\x_0$ is non-boundary.
\end{proof}

Obviously $G$ can be transformed into a grid diagram $G_0$, whose $O$'s lie in its anti-diagonal, by applying some sequence of \textit{commutation moves} (moves that swap two adjacent columns). Now reverse the sequence. Following the arguments given in \cite[Section 3.1]{MOST07grid2}, each commutation move $G_i \rightarrow G_{i+1}$ is assigned a chain homotopy equivalence map $\Phi: C(G_i) \rightarrow C(G_{i+1})$. Starting from $\x_0 \in C_0(G_0)$, we obtain one desired cycle in $C_0(G)$ by sequentially applying the corresponding chain maps. Here we explain the explicit computation of $\Phi$.

Given two diagrams that are related by a single commutation move, we depict the move by drawing the two diagrams on the same grid, where the intermediate circle $\beta$ is to be replaced with a different circle $\gamma$. By perturbing $\gamma$ so that it intersects $\beta$ transversally at two points that do not lie in the horizontal circles, the chain map $\Phi$ is given by 
\[
    \Phi(\x) = \sum_{\y \in S(G')}\sum_{\Pi \in \mathrm{Pent}^\circ (\x, \y)} U_1^{\epsilon_1} \cdots U_n^{\epsilon_n} \mathbf{y}
\]
where $\mathrm{Pent}^\circ(\x, \y)$ denotes the set of empty pentagons connecting $\x$ to $\y$, and for each such pentagon $\Pi$, the exponent $\epsilon_i \in \{0, 1\}$ is given by the number of intersections of $\Pi$ and $O_i$. See \cite{MOST07grid2} for precise definition. 

The actual computation is easy. Suppose we are swapping the $i$-th and the $(i + 1)$-th columns, so that $O_i$ and $O_{i+1}$ interchanges. Then $\epsilon_i$ is always $0$, and $\epsilon_{i+1} = 1$ if and only if, before the commutation move, $O_i$ is located at the left-upper of $O_{i + 1}$ and its vertical coordinate is between the lower and upper sides of $\Pi$. This can be seen from \Cref{fig:pentagon}, where in each of the four cases $\beta$ is drawn as a straight vertical blue line, $\gamma$ is drawn as a curve close to $\beta$. The intersection of the two curves is drawn as a small gray disk, and the other four small disks form the vertices of the pentagon $\Pi$. Note that we can ignore the $X$'s since they are not involved in both $\partial$ and $\Phi$. For the remaining exponents $\epsilon_j$, it is easy to check for each $O_j$ whether it is contained in $\Pi$.

\begin{figure}[t]
    \centering
    \input{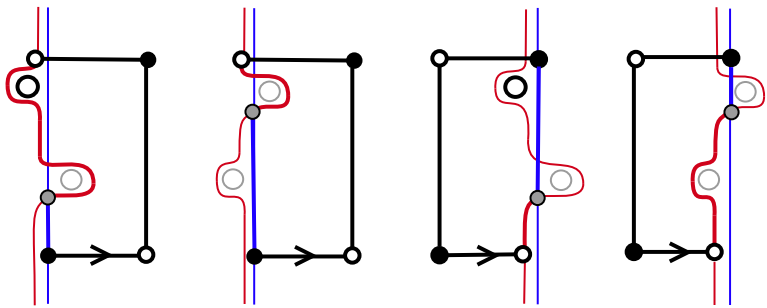}
    \caption{Relative positions of $O_i, O_{i + 1}$ and the pentagon $\Pi$}
    \label{fig:pentagon}
\end{figure}

    \subsection{Determining the realizability} \label{subsec:realizability-by-linear-system}

Having prepared the inflated generators and one homological generator $z$, we are ready to determine the realizability of the candidate regions. From \Cref{cor: tG_0 and C^-(G)}, a closed region $R$ of shift number $s$ is realizable (i.e. $R \in \tG_0(K)$) if and only if there is a cycle that is homologous to $(U_1)^s z$ and is contained in $C_{R[s]}$. This is equivalent to the condition that (the projection of) $(U_1)^s z$ is null homologous in the quotient complex $Q_R := C/C_{R[s]}$. For any $k \in \Z$, we can take a finite basis for $(Q_R)_k$ from the inflated generators contained in $C_k$ and modding out those contained in $(C_{R[s]})_k$. Having fixed such bases in degree $-2s$ and $-2s + 1$, the differential $\partial: (Q_R)_{-2s + 1} \rightarrow (Q_R)_{-2s}$ is represented by a matrix $A$. Thus the realizability of $R$ is equivalent to the existence of a solution $x$ of the linear system
\[
    A x = b
\]
where $b$ is the vector corresponding to $(U_1)^s z \in (Q_R)_{-2s}$. Now that we have reduced the problem to linear systems, the remaining task is purely computational. Here we only mention that the matrix $A$ is generally sparse, and becomes extremely large as $n$ increases. More details on sparse linear systems is given in \Cref{subsec:sparse-linear-system}. 
\begin{remark}
    This procedure of reducing a filtration-level type problem to linear systems can be applied to computing some other filtration-level type invariants. For example, Rasmussen's $s$-invariant (\cite{rasmussen2010}), which is an integer valued knot invariant defined by a $\Z$-filtration on (a deformed version of) Khovanov homology over $\mathbb{Q}$, can be computed by the same method. 
\end{remark}
    \subsection{Integration}

The four procedures are integrated into a program that goes through the following six steps:
\begin{enumerate}
    \item Enumerate the candidate regions.
    \item Setup the inflated generators.
    \item Compute one homological generator.
    \item Compute $\tau(K)$.
    \item Determine the simple regions in $\G_0(K)$.
    \item Check the remaining candidate regions.
\end{enumerate}
Assuming that we have gone through the first three steps, we explain how remaining three work, together with some techniques for reducing the computational cost. Before moving on, we remark that determining the realizability of a candidate region may contribute to discarding some other candidates. Namely, if we find a realizable region $R$, then we may discard candidates that include $R$, since they are obviously realizable and non-minimal. Conversely, if we find a non-realizable $R$, then we discard candidates that are included in $R$, since they are obviously non-realizable. The program terminates at any step if the candidates become empty.

\subsubsection{Computing $\tau(K)$}

From \Cref{thm: G_0 and others}, we know that $\tau(K)$ is given by the minimum integer $j$ such that $R_{(-1, \infty)} \cup R_{(0, j)}$ is realizable and $R_{(-1, \infty)} \cup R_{(0, j - 1)}$ is not. Here if $\tau(K) = 0$, there is a chance of an early exit: if $R_{(0, 0)}$ is realizable for both $K$ and $K^*$, then from \Cref{thm: duality} we immediately have $\G_0(K) = \{ R_{(0, 0)} \}$. In the following we assume that $\tau(K) \geq 0$, by replacing $G$ with its mirror if necessary. 

\subsubsection{Determining the simple regions in $\G_0(K)$}

Next we narrow down the candidates by determining the simple regions in $\G_0(K)$. Starting from $R_{(0, g_4)}$, we check all simple candidate regions. Note that a simple region $R = R_{(i, j)}$ belongs to $\G_0(K)$ if and only if $R$ is realizable and $R \setminus \{(i, j)\}$ is not. This strategy actually works, as we see from the results that for prime knots $K$ with up to $11$ crossings and $\tau(K) \geq 0$, $\G_0(K)$ consist only of regions whose corners lie in the first quadrant, and in many cases they are all simple. 

\subsubsection{Checking the remaining candidates}

Having determined the simple regions in $\G_0(K)$, we expect that the remaining candidates are all non-realizable (we still assume $\tau(K) \geq 0$). Suppose $R$ is the maximal remaining candidate. If we succeed to show that $R$ is non-realizable, then we are done. However such computation tends to be highly expensive, since $R$ has relatively high shift number. To avoid such computation whenever possible, it is useful to consider the mirror simultaneously. For any subset $R \subset \Z^2$, we denote by $-R$ the subset of $\Z^2$ given by 
\[
    -R := \{ (i, j) \in \Z^2 \mid (-i, -j) \in R \}.
\]

\begin{lemma}
    For any pair of realizable regions $R$ of $K$ and $R'$ of $K^*$, we have $R \cap (-R') \neq \varnothing$. 
\end{lemma}

\begin{lemma}
    Suppose $\mathcal{S}$ is a set of closed regions that includes $\G_0(K)$. If a closed region $R'$ intersects with $-R$ for any $R \in \mathcal{S}$, then $R'$ is a realizable region of $K^*$.
\end{lemma}

Thus the intermediate results obtained for $K$ can be reflected to those of $K^*$ and vice versa. After having determined the simple regions in $\G_0(K)$, we expect that the maximal remaining region of $K$ is non-realizable, and the minimal remaining region of $K^*$ is realizable. We select the one with smaller computational cost (which is simply estimated by the size of the corresponding matrix). This final process is continued until the candidates are empty.

\subsubsection{Rejection by a smaller system}
As we have seen, there are cases where we expect that a specific region $R$ is non-realizable. Suppose there is a smaller chain complex $C'$ and a $\Z^2$-filtration preserving chain map $f: C \rightarrow C'$. The existence of $c \in C$ such that 
\[
    z - \partial c \in C_R
\]
implies the existence of $c' \in C'$ such that 
\[
    f(z) - \partial c' \in C'_R.
\]
Hence by contraposition, the non-realizability of $R$ can possibly be detected using a smaller system. Our program use the chain map induced by the following projection:
\[
    \pi: \F[U_1, \cdots, U_n] \rightarrow \F[U_1, U_2],
    \quad 
    U_1 \mapsto U_1,
    \ 
    U_{\geq 2} \mapsto U_2.
\]
Further reduction techniques on the level of linear systems is described in \Cref{subsec:sparse-linear-system}.

    \section{Computational results}
\label{sec: results}

Here we list the computational results performed for prime knots of crossing number up to 11, that are both non-slice and homologically thick (with respect to $\widehat{HFK}$). The reason for this selection is that (i) if a knot $K$ is slice then $\G_0(K) = \{R_{(0, 0)}\}$, and (ii) if $K$ is homologically thin then $\G_0(K) = \G_0(\tau(K)[T_{2, 3}])$, so the non-obvious ones are those complements.

We used Culler's program \texttt{Gridlink} \cite{gridlink} to produce the input grid diagrams. This program contains the preset data for prime knots with up to 12 crossings, which are extracted from Livingston's Knotinfo database \cite{Knotinfo}. It also has the \texttt{simplify} function that randomly modifies a given diagram to reduce its grid number. In addition to the input diagram data, we used values $g_3, g_4$ and $\tau$ which are read from Knotinfo, and also $\widehat{HFK}$ which are taken from \cite{BG12}. We have implemented the algorithm as described in \Cref{sec: algorithm}. With such input data set, the results listed below were obtained within an hour by running the program on a usual laptop computer\footnote{Currently we are preparing to open source the program.}.

\Cref{table:G_0-types} presents the \textit{$\G_0$-types} of the target knots. We say that $K$ has the \textit{$\G_0$-type} of a typical knot $T$, when $\G_0(K) = \G_0(T)$. We use torus knots and those cables as typical knots, for example we read that $8_{19}$ has the $\G_0$-type of the $(3, 4)$-torus knot. The listed types are the `positive side' of the knot, that is, either one of $K$ and $K^*$ having $\tau \geq 0$. In this list, we see that there are only six types that appear as those $\G_0$-types. The next \Cref{table:G_0-tau-V_k} presents the values of $\G_0(K)$, $\tau(K)$, $V_k(K)$ and $V_k(K^*)$ for such typical knots $K$. Note that $T_{2,3; 2,1}$ (the $(2, 1)$ cable of $T_{2, 3}$) is the only type in this list that possesses a non-simple minimal region $R_{(0, 1)} \cup R_{(1, 0)}$. The algorithm for computing $\G_0(K^*)$ from $\G_0(K)$ is described in \Cref{char G_0 of mirror}. Finally \Cref{table:Upsilon} presents $\Upsilon_K$ for the typical knots $K$, where $PL[(x_0, y_0), \cdots, (x_n, y_n)]$ denotes the PL-function from $[x_0, x_n]$ to $\R$ whose graph is given by linearly connecting the adjacent coordinates.

Computations are performed for another set of knots, that is, prime knots of crossing number 12 that are non-slice, homologically thick, and $\tau = 0$. As mentioned in \Cref{sec:intro}, results have shown that all such knots have $\G_0$-type of the unknot. Hence we obtain \Cref{thm: converse}. Computations for the remaining knots with $12$ crossings are under progress.

\vspace{1em}
\begin{table}[h!]
    \centering
    \begin{tabular}{ll}
        \hline
        $K$ & $\G_0$-type \\
        \hline
        $8_{19}$ & $T_{3,4}$ \\
        $9_{42}$ & $O$ \\
        $10_{124}$ & $T_{3,5}$ \\
        $10_{128}$ & $T_{3,4}$\\
        $10_{132}$ & $T_{2, 3}$ \\
        $10_{136}$ & $O$ \\
        $10_{139}$ & $T_{3,5}$ \\
        $10_{145}$ & $T_{2, 5}$ \\
        $10_{152}$ & $T_{3,5}$ \\
        $10_{154}$ & $T_{3,4}$\\
        $10_{161}$ & $T_{3,4}$\\
        $11n6$ & $O$ \\
        $11n9$ & $T_{3,4}$ \\
        $11n12$ & $T_{2, 3}$ \\
        $11n19$ & $T_{2, 3}$ \\
        \hline
    \end{tabular}
    \hspace{2em}
    \begin{tabular}{ll}
        \hline
        $K$ & $\G_0$-type \\
        \hline
        $11n20$ & $O$ \\
        $11n24$ & $O$ \\
        $11n27$ & $T_{3,4}$\\
        $11n31$ & $T_{2, 5}$ \\
        $11n34$ & $O$ \\
        $11n38$ & $O$ \\
        $11n45$ & $O$ \\
        $11n57$ & $T_{3,4}$\\
        $11n61$ & $T_{2,3; 2,1}$ \\
        $11n70$ & $T_{2, 3}$ \\
        $11n77$ & $T_{3,5}$ \\
        $11n79$ & $O$ \\
        $11n80$ & $T_{2, 3}$ \\
        $11n81$ & $T_{3,4}$ \\
        $11n88$ & $T_{3,4}$\\
        \hline
    \end{tabular}
    \hspace{2em}
    \begin{tabular}{ll}
        \hline
        $K$ & $\G_0$-type \\
        \hline
        $11n92$ & $O$ \\
        $11n96$ & $O$ \\
        $11n102$ & $T_{2, 3}$ \\
        $11n104$ & $T_{3,4}$ \\
        $11n111$ & $T_{2, 3}$ \\
        $11n126$ & $T_{3,4}$\\
        $11n133$ & $T_{2,3; 2,1}$ \\
        $11n135$ & $T_{2, 5}$ \\
        $11n138$ & $O$ \\
        $11n143$ & $O$ \\
        $11n145$ & $O$ \\
        $11n151$ & $T_{2, 3}$ \\
        $11n152$ & $T_{2, 3}$ \\
        $11n183$ & $T_{3,4}$ \\
        \\
        \hline
    \end{tabular}
    \caption{List of $\G_0$-types}\label{table:G_0-types}

    \vspace{1.5em}

    \begin{tabular*}{\textwidth}{l @{\extracolsep{\fill}} llll}
        \hline
        $K$ & 
            $\G_0(K)$ &
            $\tau(K)$ &
            $V_k(K)$ &
            $V_k(K^*)$ \\
        \hline
    
        $O$ & 
            $R_{(0, 0)}$ &
            0 &
            \{0\} &
            \{0\} \\

        $T_{2, 3}$ & 
            $R_{(0, 1)}, R_{(1, 0)}$ &
            1 &
            \{1, 0\} &
            \{0\} \\

        $T_{2, 5}$ & 
            $R_{(0, 2)}, R_{(1, 1)}, R_{(2, 0)}$ &
            2 &
            \{1, 1, 0\} &
            \{0\} \\

        $T_{2,3; 2,1}$ & 
            $R_{(0, 2)}, R_{(0, 1)} \cup R_{(1, 0)}, R_{(2, 0)}$ &
            2 &
            \{1, 1, 0\} &
            \{0\} \\

        $T_{3,4}$ &
            $R_{(0, 3)}, R_{(1, 1)}, R_{(3, 0)}$ &
            3 &
            \{1, 1, 1, 0\} &
            \{0\} \\

        $T_{3,5}$ & 
            $R_{(0, 4)}, R_{(1, 2)}, R_{(2, 1)}, R_{(4, 0)}$ &
            4 &
            \{2, 1, 1, 1, 0\} &
            \{0\} \\

        \hline
    \end{tabular*}
    \caption{$\G_0(K)$, $\tau(K)$, $V_k(K)$ and $V_k(K^*)$}
    \label{table:G_0-tau-V_k}
    
    \vspace{1.5em}

    \begin{tabular*}{\textwidth}{ll}
        \hline
        $K$ & 
            $\Upsilon_K$ \\
        \hline
    
        $O$ & 
            $0$ \\
            
        $T_{2, 3}$ & 
            $PL[(0, 0), (1, -1), (2, 0)]$ \\
            
        $T_{2, 5}$ & 
            $PL[(0, 0), (1, -2), (2, 0)]$ \\
    
        $T_{2,3; 2,1}$ & 
            $PL[(0, 0), (2/3, -4/3), (1, -1), (4/3, -4/3), (2, 0)]$ \\

        $T_{3,4}$ &
            $PL[(0, 0), (2/3, -2), (4/3, -2), (2, 0)]$ \\
            
        $T_{3,5}$ & 
            $PL[(0, 0), (2/3, -8/3), (1, -3), (4/3, -8/3), (2, 0)]$ \\

        \hline
    \end{tabular*}
    \caption{$\Upsilon_K$}\label{table:Upsilon}
\end{table}

    \newpage
    \appendix
\section{The duality theorem for $\G_0$}
\label{sec: duality}
In this appendix, we prove the duality theorem for $\G_0$,
 stated as \Cref{thm: duality}.
To prove the theorem, we first introduce the notion of {\it the dual} of a closed region $R$ (denoted $R^*$). Then, for any $C = CFK^\infty(K)$, it is proved that the subcomplex $C^*_R$ consists of chains which
map $C_{R^*}$ to zero.
In \Cref{subsec: proof of duality}, 
we give a proof of \Cref{thm: duality}.
In \Cref{subsec: computing G_0(K^*) from G_0(K)},
we provide a method for computing $\G_0(K^*)$ from $\G_0(K)$
algorithmically.
\subsection{The dual of a closed region}
For any given $R \in \CR(\z^2)$, {\it the dual region $R^*$ of $R$} is defined by
$$
R^* := \z^2 \setminus (-R) = \{(i,j) \in \z^2 \mid (-i,-j) \not\in R\}.
$$
Here we show several basic properties of dual regions.
\begin{lemma}
$R^* \in \CR(\z^2)$.
\end{lemma}

\begin{proof}
Suppose that $(k,l) \in \z^2$, $(i,j) \in R^*$ and $(k,l) \leq (i,j)$.
Under these assumptions, if $(k,l) \neq R^*$, then
$(-k,-l) \in R$. On the other hand, since $(-i,-j) \leq (-k,-l)$, we have $(-i,-j) \in R$.
This contradicts to $(i,j) \in R^*$, and hence $(k,l) \in R^*$.
\end{proof}

\begin{lemma}
$R^{**}=R$.
\end{lemma}
\begin{proof}
$(i,j) \in R^{**} \Leftrightarrow (-i,-j) \neq R^* \Leftrightarrow (i,j) \in R$.
\end{proof}

\begin{lemma}
$R^* =  \bigcap_{(k,l)\in R} 
\{ i \leq -k-1 \text{ {\rm or }} j \leq -l-1 \}$.
\end{lemma}

\begin{proof}
We first prove $R^* \subset  \bigcap_{(k,l)\in R} 
\{ i \leq -k-1 \text{ {\rm or }} j \leq -l-1 \}$.
Let $(i,j) \in R^*$. Then, for any $(k,l) \in R$, either $-i>k$ or $-j > l$ holds.
Equivalently, we have $i \leq -k-1$ or $j \leq -l-1$.

Next, we prove $R^* \supset  \bigcap_{(k,l)\in R} 
\{ i \leq -k-1 \text{ {\rm or }} j \leq -l-1 \}$.
Let $(i,j) \in \cap_{(k,l) \in R}\{ i \leq -k-1 \text{ {\rm or }} j \leq -l-1 \}$.
Then, for any $(k,l) \in R$, either $-i > k$ or $-j > l$ holds.
In particular, we have $(-i,-j) \neq (k,l)$, and hence 
$(-i,-j) \not\in R$.
\end{proof}
Using the above lemmas, we have the following proposition,
which plays an essential role in the proof of 
\Cref{thm: duality}.
\begin{proposition}
\label{dual complex and region}
Let $C := CFK^\infty(K)$. Then, for any $R \in \CR(\z^2)$, 
we have
\[
C^*_R = \{ \psi \in C^* \mid \varepsilon \circ \psi (C_{R^*}) = 0\}.
\]
\end{proposition}

\begin{proof}
We first prove 
$C^*_R \subset \{ \psi \in C^* \mid \varepsilon \circ \psi (C_{R^*}) = 0\}$.
Indeed, this follows from the inequalities
\begin{eqnarray*}
C^*_R &=& \sum_{(k,l) \in R} \falg{k}(C^*) \cap \falex{l}(C^*)\\
 &=& \sum_{(k,l) \in R} \left\{ \psi \in C^* \  \middle| \ \varepsilon \circ 
\psi (\falg{-k-1} + \falex{-l-1})=0 \right\} \\
&=& \sum_{(k,l) \in R} \left\{ \psi \in C^* \  \middle| \ \varepsilon \circ 
\psi (C_{\{i \leq -k-1 \text{ {\rm or} } j \leq -l-1\}})=0 \right\} \\
&\subset& \left\{ \psi \in C^* \  \middle| \ \varepsilon \circ 
\psi (C_{\bigcap_{(k,l) \in R} \{i \leq -k-1 \text{ {\rm or} } j \leq -l-1\}})=0 \right\}\\
&=& \left\{ \psi \in C^* \  \middle| \ \varepsilon \circ 
\psi (C_{R^*})=0 \right\}.
\end{eqnarray*}

Next, we prove
$C^*_R \supset \{ \psi \in C^* \mid \varepsilon \circ \psi (C_{R^*}) = 0\}$.
Take a filtered basis $\{x_s\}_{1 \leq s \leq r}$ for $C$, and then
any $\psi \in C^*$ is written as 
$\psi = \sum_{1 \leq s \leq r, \ t \in \Z} a_{s,t} U^t x^*_s$,
where $a_{s,t}=0$ for all but finitely many $(s,t)$. 
Now, suppose that $\varepsilon \circ \psi (C_{R^*})= \varepsilon \circ 
\psi (C_{\bigcap_{(k,l) \in R} \{i \leq -k-1 \text{ {\rm or} } j \leq -l-1\}})=0$.
Then, we see that $U^{-t}x_s \not\in C_{R^*}$ for any $(s,t)$ with $a_{s,t} \neq 0$.
This implies that 
$R_{U^{-t}x_s}=R_{(\Alg(U^{-t}x_s), \Alex(U^{-t}x_s))} \not\in R^*$, and hence
$(\Alg(U^{-t}x_s), \Alex(U^{-t}x_s)) \geq (-k,-l)$
for some $(k,l) \in R$. In particular, we have 
$$U^t x_s^*(\falg{-k-1}+\falex{-l-1})=0$$ 
for such $(k,l)$. Consequently, we have
$$
\psi = \sum_{a_{s,t}\neq 0} a_{s,t} U^t x^*_s
\in  \sum_{(k,l) \in R} \left\{ \psi \in C^* \  \middle| \ \varepsilon \circ 
\psi (\falg{-k-1}+\falex{-l-1})=0 \right\} = C^*_R.
$$
\end{proof}

\subsection{Proof of \Cref{thm: duality}}
\label{subsec: proof of duality}
Now we prove \Cref{thm: duality}.
\renewcommand{\thesection}{\arabic{section}}
\setcounter{section}{2}
\setcounter{proposition}{15}
\begin{theorem}
For any knot $K$, the equalities
$$
\tG_0(K^*) = \{ R \in \CR(\z^2) \mid 
\forall R' \in \G_0(K),\ R \cap (-R') \neq \varnothing \}
$$
and
$$
\G_0(K^*) = \min\{ R \in \CR(\z^2) \mid 
\forall R' \in \G_0(K),\ R \cap (-R') \neq \varnothing \}
$$
hold. 
\end{theorem}
\renewcommand{\thesection}{\Alph{section}}
\setcounter{section}{1}
\setcounter{proposition}{4}
 The proof is an analogy of \cite[Lemma 2.28]{2019arXiv190709116S}.
\begin{proof}
By the definition of $\G_0(K)$, it is sufficient to prove the first equality. We denote by $\G_0(K)^{\perp}$ 
the right hand side of the desired equality; namely, we set
\[
\G_0(K)^{\perp} := \{ R \in \CR(\z^2) \mid 
\forall R' \in \G_0(K),\ R \cap (-R') \neq \varnothing \}.
\]
Then, by the definition of dual regions, it is obvious that
\[
\G_0(K)^{\perp} = \{ R \in \CR(\z^2) \mid 
\forall R' \in \G_0(K),\ R' \not\subset R^*\}.
\]

We first prove $\tG_0(K^*) \subset \G_0(K)^{\perp}$.
Let $R \in \tG_0(K^*)$ and $C:=CFK^\infty(K)$. 
Then, by \Cref{dual complex and region},
there exists a homological generator $\varphi \in C^*_0$
lying in 
$$
C^*_R = \left\{ \psi \in C^* \  \middle| \ \varepsilon \circ 
\psi (C_{R^*})=0 \right\}.
$$
In particular, we have
$\varepsilon \circ \varphi(C_{R^*}) =0$, 
and $\varepsilon \circ \varphi$ is
decomposed as $\varepsilon \circ \varphi = \widetilde{\varphi} \circ p$
where $\widetilde{\varphi} \in \Hom_{\F}(C_{R^*}, \F)$ 
is a cocycle and $p : C \to C/C_{R^*}$ 
is the projection.
Now, let $R' \in \G_0(C)$ and $x$
be a homological generator
whose closure gives $R'$.
Then we have $\widetilde{\varphi}(p (x))= (\varepsilon \circ \varphi)(x)=1$.
This implies that $p(x) \neq 0$, and hence $R^* \not\supset R'$. 
Therefore, we have $R \in \G_0(K)^{\perp}$.

Next, we prove $\tG_0(K^*) \supset \G_0(K)^{\perp}$.
Suppose that 
$R \in \G_0(K)^{\perp}$. 
Then, by Theorem~\ref{thm: minimalize}, we see that $R^* \not\in \tG_0(C)$,
and hence $p_{*,0} \colon H_0(C) \to H_0(C/C_{R^*})$ is injective.
Let $x \in C_0$ be a homological generator,
and then we have $p_{*,0}([x]) \neq 0$.
In addition, $\dim_{\F}(C/C_{R^*})_0$
is finite, and hence we can take a finite $\F$-basis for 
$H_0(C/C_{R^*})$ containing $p_{*,0}([x])$.
Thus, by using the identification (given by 
\Cref{lem: dual epsilon})
$$
\Hom_{\F}(H_0(C/C_{R^*}),\F) \cong H^0(C/C_{R^*} ;\F),
$$
 we can take a cocycle 
$\psi \in \Hom_{\F}((C/C_{R^*})_0,\F)$
whose cohomology class is the dual $(p_{*,0}([x]))^*$.
Moreover, the map 
$\varepsilon_0$ in \Cref{lem: dual epsilon} is bijective,
and hence we can take the inverse $\varphi := \varepsilon^{-1}_0(\psi \circ p) \in C^*_0$.
Note that since $\varepsilon \circ \varphi(x) = \psi(p(x))=1$,
the element $\varphi \in C^*_0$ is a homological generator.
Moreover, the equalities
$$
\varepsilon \circ \varphi 
(C_{R^*}) = 
\psi \circ p(C_{R^*})
= 0
$$
hold, and hence $\varphi$ lies in $C^*_R$. This proves $R \in \tG_0(K^*)$.
\end{proof}

\subsection{Computing $\G_0(K^*)$ from $\G_0(K)$}
\label{subsec: computing G_0(K^*) from G_0(K)}
For any non-empty finite set $\mathcal{S}=\{R_i\}_{i=1}^N$
of semi-simple regions,
let 
\[
S^{\perp} := \{ R \in \CR(\z^2) \mid 
\forall R_i \in \mathcal{S},\ R \cap (-R_i) \neq \varnothing\}.
\]
Then we have the following duality theorem
between $\mathcal{S}$ and $\min \mathcal{S}^{\perp}$.
\begin{theorem}
\label{char orthonormal}
The equality
$$
\min(\mathcal{S}^{\perp}) = 
\min \left\{cl(\{-p_i\}_{i=1}^N) 
\ \middle| \  (p_1, \ldots, p_N) \in \prod_{i=1}^N c(R_i) \right\}
$$
holds.
\end{theorem}

\begin{remark}
It can happen that 
$p_i \leq p_j$ for different $i$ and $j$,
and hence the order of $\{p_i\}_{i=1}^N$ can be less than $N$. 
\end{remark}
Recall that $\G_0(K) \subset \CR^{ss}(\z^2)$ for any knot $K$,
and hence, combining with \Cref{thm: duality},
we have the following formula.
\begin{corollary}
\label{char G_0 of mirror}
Let $\G_0(K) = \{R_i\}_{i=1}^N$.
Then, the equality
$$
\G_0(K^*) = 
\min \left\{cl(\{-p_i\}_{i=1}^N) 
\ \middle| \  (p_1, \ldots, p_N) \in \prod_{i=1}^N c(R_i) \right\}
$$
holds.
In particular, $\G_0(K^*)$ is algorithmically determined from
$\G_0(K)$.
\end{corollary}

Now we prove \Cref{char orthonormal}.

\def\proofname{Proof of Theorem~\ref{char orthonormal}}
\begin{proof}
Set $\mathcal{S}' := \{cl(\{-p_i\}_{i=1}^N) 
 \mid   (p_1, \ldots, p_N) \in \prod_{i=1}^N c(R_i) \}$.
We first note that for any 
$(p_1, \ldots, p_N) \in \prod_{i=1}^N c(R_i)$,
we have
\[
-p_i \in cl(\{-p_i\}_{i=1}^N) \cap (-R_i) \neq \varnothing.
\]
This implies $\mathcal{S}' \subset \mathcal{S}^{\perp}$. 

Now we prove $\min(\mathcal{S}^{\perp}) \subset \min \mathcal{S}'$.
 Let $R \in \min(\mathcal{S}^{\perp})$.
Then $R \cap (-R_i) \neq \varnothing$ for any $1 \leq i \leq N$,
which is equivalent to $R_i \not\subset R^*$.
Here, \Cref{lem: char ss} gives $R_i = cl(c(R_i))$,
and hence there exists an element $p_i \in c(R_i)$
such that $p_i \not\in R^*$.
In particular, we have $-p_i \in R$, 
and hence $R \supset cl(\{-p_i\}_{i=1}^N)$.
Since $cl(\{-p_i\}_{i=1}^N) \in \mathcal{S}' \subset \mathcal{S}^{\perp}$, 
it follows from the minimality
of $R$ in $\mathcal{S}^{\perp}$ that 
$$
R= cl(\{-p_i\}_{i=1}^N) \in \mathcal{S}'.$$
The minimality of $R$ in $\mathcal{S}'$ directly follows from
the minimality in $\mathcal{S}^{\perp}$. Therefore, we have $R \in \min S'$.

Next, we prove $\min(\mathcal{S}^{\perp}) \supset \min \mathcal{S}'$.
Let $cl(\{-p_i\}_{i=1}^N) \in \min \mathcal{S}'$.
Then we have $cl(\{-p_i\}_{i=1}^N) \in \mathcal{S}^{\perp}$.
Suppose that $R \in \mathcal{S}^{\perp}$ satisfies $R \subset cl(\{-p_i\}_{i=1}^N)$.
Then, in a similar way to the converse case, we can find an element 
$cl(\{-q_i\}_{i=1}^N) \in \mathcal{S}'$ such that $cl(\{-q_i\}_{i=1}^N) \subset R$.
In particular, we have
$$
cl(\{-q_i\}_{i=1}^N) \subset R \subset cl(\{-p_i\}_{i=1}^N),
$$
and hence the minimality of $cl(\{-p_i\}_{i=1}^N)$ in $\mathcal{S}'$ gives
$$
cl(\{-q_i\}_{i=1}^N) = R = cl(\{-p_i\}_{i=1}^N).
$$
This implies $cl(\{-p_i\}_{i=1}^N) \in \min (\mathcal{S}^{\perp})$.
\end{proof}
\def\proofname{Proof}

    \section{On sparse linear systems} \label{subsec:sparse-linear-system}

Here we briefly discuss computational methods for handling large sparse linear system over any computational field, such as $\mathbb{Q}$ or $\F_p$. There are essentially two families of algorithms for handling linear systems: direct methods and iterative methods. Given a matrix $A$, direct methods (such as Gaussian elimination and LU factorization) perform elementary operations on $A$ so that information can be read directly from the modified matrix $A'$, whereas iterative methods (such as the Wiedemann algorithm) multiply $A$ repeatedly on a randomly generated vector seeking for a solution. The implementation of our program is based on the former.

\textit{LU factorization} is one of the standard methods for handling (dense) linear systems. Suppose $A$ is an arbitrary $n$-by-$m$ matrix of rank $r$. It is known that, after some permutations of rows and columns, $A$ can be factored into a product of two matrices $L$ and $U$, where $L$ is a lower triangular $n$-by-$r$ matrix and $U$ is an upper triangular $r$-by-$m$ matrix. To be precise, there exists some permutation matrices $P$ and $Q$ such that
\[
    P A Q = L U
\]
holds (in some literature it is expressed in the form $A = PLUQ$ and called the $PLUQ$ factorization of $A$). This can be seen as factoring a linear map represented by $A$ into a composition of a surjection (represented by $U$) and an injection (represented by $L$). Having computed $PAQ = LU$, solving a linear system $Ax = b$ breaks up into solving $Ly = Pb$ and $U(Q^{-1}x) = y$. The former equation may or may not have a solution, and in case it does the solution is unique. The latter equation always has a solution with degree of freedom $m - r$. Both equations can be solved easily from the structure of $L$ and $U$. In the following, for simplicity, we omit the permutations $P, Q$ and write $A \sim LU$. 

LU factorization is obtained from Gaussian elimination. Suppose $A = (a_{ij})_{1 \leq i \leq n,\ 1 \leq j \leq m}$ is non-zero. Then there exists some non-zero component of $A$, which we may assume that it is $a_{11}$. We may write
\[
    A \sim 
    \begin{pmatrix}
    1 \\
    a_{21} / a_{11} \\
    \vdots \\
    a_{n1} / a_{11} \\
    \end{pmatrix}
    \begin{pmatrix}
    a_{11} & a_{12} & \cdots & a_{1m}
    \end{pmatrix}
    +
    \begin{pmatrix}
    0 & 0 & \cdots & 0 \\ 
    0 & \\
    \vdots & & \mbox{\large $A'$}\\ 
    0 & 
    \end{pmatrix}.
\]
If $A' = (a'_{ij})_{2 \leq i \leq n,\ 2 \leq j \leq m}$ is non-zero, then we repeat the same process and obtain
\[
    A \sim
    \begin{pmatrix}
    1 & 0 \\
    a_{21} / a_{11} & 1\\
    \vdots & \vdots \\
    a_{n1} / a_{11} & a'_{n2} / a'_{22}\\
    \end{pmatrix}
    \begin{pmatrix}
    a_{11} & a_{12} & \cdots & a_{1m} \\
    0 & a'_{22} & \cdots & a'_{2m}
    \end{pmatrix}
    +
    \begin{pmatrix}
    0 & 0 & \cdots & 0 \\ 
    0 & 0 & \cdots & 0 \\ 
    \vdots & \vdots & \mbox{\large $A''$}\\ 
    0 & 0
    \end{pmatrix}
\]
Repeating this process until the remaining block becomes zero yields the desired factorization. Note that at each step, permutations of the complementary submatrix preserves the structure of the partially computed $L$ and $U$. This process can also be seen as iteratively computing the \textit{Schur complement} of the upper left non-zero component. In general, for a block matrix 
\[
    X = 
    \begin{pmatrix} A & B \\ C & D \end{pmatrix}
\]
with $A$ non-singular, the \textit{Schur complement} of $A$ is defined by 
\[
    S := D - C A^{-1} B.
\]

Now we assume that the matrix $A$ is large and sparse, say $n, m > 500,000$ and its density is below $0.01\%$. In such case it is impractical to calculate its LU factorization directly. For this problem, Bouillaguet, Delaplace and Voge \cite{bouillaguet2017parallel} propose a method of finding a large amount of \textit{structural pivots}, and the algorithm is implemented in their program \texttt{SpaSM} \cite{spasm}. Structural pivots are non-zero entries of $A$ such that after some permutations $A$ is transformed into a matrix of the form
\[
    \begin{pmatrix}
    U & B \\
    C & D
    \end{pmatrix}
\]
where $U$ is a regular upper triangular matrix with pivots lying in its diagonal. Given a set of structural pivots, we immediately obtain a partial LU factorization of $A$ as
\[
    A \sim
    \begin{pmatrix}
    I \\ C'
    \end{pmatrix}
    \begin{pmatrix}
    U & B
    \end{pmatrix}
    +
    \begin{pmatrix}
    O & O \\
    O & S
    \end{pmatrix}
\]
where $C' = CU^{-1}$, and $S = D - C U^{-1} B$ is the Schur complement of $U$. To complete the LU factorization of $A$, we either iterate the same process on $S$, or switch to dense LU factorization, depending on the size and density of $S$. 

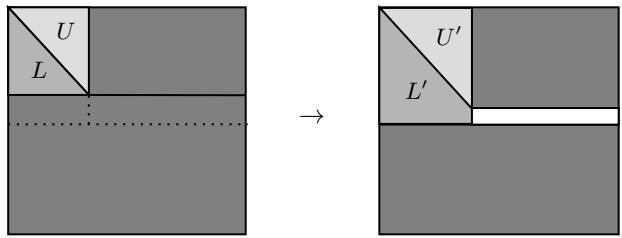
\begin{figure}[t]
    \centering
    \tikzset{every picture/.style={line width=0.75pt}} 

\begin{tikzpicture}[x=0.75pt,y=0.75pt,yscale=-1,xscale=1]

\draw  [fill={rgb, 255:red, 128; green, 128; blue, 128 }  ,fill opacity=1 ] (80.11,12) -- (199.91,12) -- (199.91,126.58) -- (80.11,126.58) -- cycle ;
\draw  [fill={rgb, 255:red, 221; green, 221; blue, 221 }  ,fill opacity=1 ] (120.8,56.24) -- (80.11,12) -- (120.8,12) -- cycle ;
\draw  [fill={rgb, 255:red, 181; green, 181; blue, 181 }  ,fill opacity=1 ] (80.11,12) -- (120.8,56.24) -- (80.11,56.24) -- cycle ;
\draw    (120.8,56.24) -- (199.91,56.37) ;
\draw  [dash pattern={on 0.84pt off 2.51pt}]  (80,71.05) -- (200.99,71.05) ;
\draw  [fill={rgb, 255:red, 128; green, 128; blue, 128 }  ,fill opacity=1 ] (267.29,12) -- (388,12) -- (388,126.58) -- (267.29,126.58) -- cycle ;
\draw  [fill={rgb, 255:red, 220; green, 220; blue, 220 }  ,fill opacity=1 ] (314,70.53) -- (267.58,12.31) -- (314,12.31) -- cycle ;
\draw  [fill={rgb, 255:red, 181; green, 181; blue, 181 }  ,fill opacity=1 ] (267.29,12) -- (321,70.84) -- (267.29,70.84) -- cycle ;
\draw    (267.29,71.2) -- (388,71.2) ;
\draw  [dash pattern={on 0.84pt off 2.51pt}]  (120.8,71.05) -- (120.8,55.36) ;
\draw  [draw opacity=0] (11,10.83) -- (26,10.83) -- (26,27.27) -- (11,27.27) -- cycle ;
\draw  [fill={rgb, 255:red, 255; green, 255; blue, 255 }  ,fill opacity=1 ] (314,62.84) -- (388,62.84) -- (388,71.2) -- (314,71.2) -- cycle ;

\draw (697,131) node    {$\rightarrow $};
\draw (109.36,23.69) node  [font=\small]  {$U$};
\draw (95.67,43.49) node  [font=\small]  {$L$};
\draw (302.65,26.49) node  [font=\small]  {$U'$};
\draw (286.28,53.55) node  [font=\small]  {$L'$};
\draw (233.34,68.13) node    {$\rightarrow $};

\end{tikzpicture}
    \caption{Incremental LU factorization}
    \label{fig:incremental-LU}
\end{figure}{}

The heavy part of the above described method is in the computation of the Schur complement. For our purpose, recall from \Cref{subsec:realizability-by-linear-system} that we only need to tell whether the system has a solution or not. Assuming that we have obtained pivots whose number is close to the rank of $A$, only a few more independent columns might suffice to find one solution. Likewise, only a few more independent rows might suffice to detect that the system has no solution. Thus the approach we take is the following: after we have obtained the pivots and permuted the matrix correspondingly, we divide the remaining columns (resp. rows) into small chunks and incrementally expand the factorization by performing the factorization on the submatrices (see \Cref{fig:incremental-LU}). At each step we check whether the subsystem has a solution or not. This enables us to exit the computation in an earlier stage whenever possible.

    \bibliographystyle{plain}
    \bibliography{bibliography.bib}

\begin{thebibliography}{10}

\bibitem{BG12}
John~A. Baldwin and William~D. Gillam.
\newblock Computations of {H}eegaard-{F}loer knot homology.
\newblock {\em J. Knot Theory Ramifications}, 21(8):1250075, 65, 2012.

\bibitem{bouillaguet2017parallel}
Charles Bouillaguet, Claire Delaplace, and Marie{-}Emilie Voge.
\newblock Parallel sparse {PLUQ} factorization modulo p.
\newblock In Jean{-}Charles Faug{\`{e}}re, Michael~B. Monagan, and
  Hans{-}Wolfgang Loidl, editors, {\em Proceedings of the International
  Workshop on Parallel Symbolic Computation, PASCO@ISSAC 2017, Kaiserslautern,
  Germany, July 23-24, 2017}, pages 8:1--8:10. {ACM}, 2017.

\bibitem{Knotinfo}
Jae~Choon Cha and Charles Livingston.
\newblock Knot{I}nfo: {T}able of {K}not {I}nvariants.
\newblock \url{http://www.indiana.edu/~knotinfo}.

\bibitem{spasm}
The~SpaSM group.
\newblock {\em {SpaSM}: a Sparse direct Solver Modulo $p$}, v1.2 edition, 2017.
\newblock \url{http://github.com/cbouilla/spasm}.

\bibitem{heap1963permutations}
B.~R. Heap.
\newblock Permutations by interchanges.
\newblock {\em The Computer Journal.}, 6(3):293--298, 1963.

\bibitem{Hom17}
Jennifer Hom.
\newblock A survey on {H}eegaard {F}loer homology and concordance.
\newblock {\em J. Knot Theory Ramifications}, 26(2):1740015, 24, 2017.

\bibitem{HW16}
Jennifer Hom and Zhongtao Wu.
\newblock Four-ball genus bounds and a refinement of the {O}zv\'ath-{S}zab\'o
  tau invariant.
\newblock {\em J. Symplectic Geom.}, 14(1):305--323, 2016.

\bibitem{KP18}
Min~Hoon Kim and Kyungbae Park.
\newblock An infinite-rank summand of knots with trivial {A}lexander
  polynomial.
\newblock {\em J. Symplectic Geom.}, 16(6):1749--1771, 2018.

\bibitem{Liv17}
Charles Livingston.
\newblock Notes on the knot concordance invariant upsilon.
\newblock {\em Algebr. Geom. Topol.}, 17(1):111--130, 2017.

\bibitem{MOS09grid1}
Ciprian Manolescu, Peter Ozsv\'{a}th, and Sucharit Sarkar.
\newblock A combinatorial description of knot {F}loer homology.
\newblock {\em Ann. of Math. (2)}, 169(2):633--660, 2009.

\bibitem{MOST07grid2}
Ciprian Manolescu, Peter Ozsv\'{a}th, Zolt\'{a}n Szab\'{o}, and Dylan Thurston.
\newblock On combinatorial link {F}loer homology.
\newblock {\em Geom. Topol.}, 11:2339--2412, 2007.

\bibitem{gridlink}
{Marc Culler}.
\newblock Gridlink: A tool for knot theorists.
\newblock \url{http://homepages.math.uic.edu/~culler/gridlink/}.

\bibitem{NW15}
Yi~Ni and Zhongtao Wu.
\newblock Cosmetic surgeries on knots in {$S^3$}.
\newblock {\em J. Reine Angew. Math.}, 706:1--17, 2015.

\bibitem{OS03grading}
Peter Ozsv\'{a}th and Zolt\'{a}n Szab\'{o}.
\newblock Absolutely graded {F}loer homologies and intersection forms for
  four-manifolds with boundary.
\newblock {\em Adv. Math.}, 173(2):179--261, 2003.

\bibitem{OS03tau}
Peter Ozsv\'ath and Zolt\'an Szab\'o.
\newblock Knot {F}loer homology and the four-ball genus.
\newblock {\em Geom. Topol.}, 7:615--639, 2003.

\bibitem{OS04knot}
Peter Ozsv\'ath and Zolt\'an Szab\'o.
\newblock Holomorphic disks and knot invariants.
\newblock {\em Adv. Math.}, 186(1):58--116, 2004.

\bibitem{OSS15gridbook}
Peter~S Ozsv{\'a}th, Andr{\'a}s~I Stipsicz, and Zolt{\'a}n Szab{\'o}.
\newblock {\em Grid homology for knots and links}, volume 208.
\newblock American Mathematical Soc., 2015.

\bibitem{OSS17}
Peter~S. Ozsv\'{a}th, Andr\'{a}s~I. Stipsicz, and Zolt\'{a}n Szab\'{o}.
\newblock Concordance homomorphisms from knot {F}loer homology.
\newblock {\em Adv. Math.}, 315:366--426, 2017.

\bibitem{OS11rational}
Peter~S. Ozsv\'ath and Zolt\'an Szab\'o.
\newblock Knot {F}loer homology and rational surgeries.
\newblock {\em Algebr. Geom. Topol.}, 11(1):1--68, 2011.

\bibitem{Pet13}
Ina Petkova.
\newblock Cables of thin knots and bordered {H}eegaard {F}loer homology.
\newblock {\em Quantum Topol.}, 4(4):377--409, 2013.

\bibitem{Pic18pre}
Lisa {Piccirillo}.
\newblock {The Conway knot is not slice}.
\newblock {\em arXiv e-prints}, page arXiv:1808.02923, Aug 2018, to appear in
  {\it Ann.\ of Math.}

\bibitem{rasmussen2010}
Jacob Rasmussen.
\newblock Khovanov homology and the slice genus.
\newblock {\em Invent. Math.}, 182(2):419--447, 2010.

\bibitem{rasmussen2003floer}
Jacob~Andrew Rasmussen.
\newblock {\em Floer homology and knot complements}.
\newblock ProQuest LLC, Ann Arbor, MI, 2003.
\newblock Thesis (Ph.D.)--Harvard University.

\bibitem{2019arXiv190709116S}
Kouki {Sato}.
\newblock {The $\nu^+$-equivalence classes of genus one knots}.
\newblock {\em arXiv e-prints}, page arXiv:1907.09116, Jul 2019.

\bibitem{Yas15pre}
Kouichi {Yasui}.
\newblock {Corks, exotic 4-manifolds and knot concordance}.
\newblock {\em arXiv e-prints}, page arXiv:1505.02551, May 2015.

\end{thebibliography}

\end{document}